\documentclass{amsart}
\usepackage{latexsym,amsxtra,amscd,ifthen}
\usepackage{amsfonts}
\usepackage{verbatim}
\usepackage{amsmath}
\usepackage{amsthm}
\usepackage{amssymb}

\numberwithin{equation}{section}

\theoremstyle{plain}
\newtheorem{theorem}{Theorem}[section]
\newtheorem{lemma}[theorem]{Lemma}
\newtheorem{proposition}[theorem]{Proposition}

\newtheorem{construction}[theorem]{Construction}

\theoremstyle{definition}

\newtheorem{remark}[theorem]{Remark}
\newtheorem*{remark*}{Remark}
\newtheorem{question}[theorem]{Question}
\newtheorem{subsec}[theorem]{}

\makeatletter              
\let\c@equation\c@theorem  
\makeatother

\newcommand{\fg}{\mathfrak g}
\newcommand{\fh}{\mathfrak h}

\DeclareMathOperator{\Spec}{Spec}
\DeclareMathOperator{\height}{height}
\DeclareMathOperator{\gldim}{gldim}

\DeclareMathOperator{\Ext}{Ext}

\DeclareMathOperator{\gr}{gr}

\DeclareMathOperator{\Aut}{Aut}
\DeclareMathOperator{\Kdim}{Kdim}

\DeclareMathOperator{\injdim}{injdim}

\DeclareMathOperator{\GKdim}{GKdim}

\DeclareMathOperator{\Hom}{Hom}

\newcommand{\fm}{\mathfrak{m}}

\newcommand{\id}{\operatorname{id}}

\newcommand{\ZZ}{{\mathbb Z}}

\DeclareMathOperator{\Fract}{Fract}

\newcommand\cO{{\mathcal O}}

\newcommand\xbar{\overline{x}}
\newcommand\ybar{\overline{y}}
\newcommand\Hbar{\overline{H}}
\newcommand\Deltabar{\overline{\Delta}}

\newcommand\Hzr{{\rm($H$)}}

\newcommand\nat{{\rm($\natural$)}}
\newcommand\Hzrnat{{\rm($H\natural$)}}

\newcommand\co{\operatorname{co}}
\newcommand\chr{\operatorname{char}}

\newcommand\eps{\epsilon}
\newcommand\Znonneg{\ZZ_{\ge0}}
\newcommand\Zpos{\ZZ_{>0}}
\newcommand\kx{k^{\times}}
\newcommand\gfrak{{\mathfrak g}}

\newcommand{\phihat}{\hat{\phi}}
\DeclareMathOperator{\io}{io}
\DeclareMathOperator{\pideg}{p.i.deg}

\begin{document}

\title[Hopf algebras of GK-dimension two]
{Noetherian Hopf algebra domains of Gelfand-Kirillov 
dimension two}

\author{K.R. Goodearl and J.J. Zhang}

\address{Goodearl: Department of Mathematics 
University of California at Santa Barbara,
Santa Barbara, CA 93106}

\email{goodearl@math.ucsb.edu} 

\address{Zhang: Department of Mathematics, Box 354350,
University of Washington, Seattle, Washington 98195, USA}

\email{zhang@math.washington.edu}

\begin{abstract}
We classify all noetherian Hopf algebras $H$ over an algebraically 
closed field $k$ of characteristic zero which are integral domains 
of Gelfand-Kirillov dimension two and satisfy the condition 
$\Ext^1_H(k,k)\neq 0$. The latter condition is conjecturally 
redundant, as no examples are known (among noetherian Hopf algebra 
domains of GK-dimension two) where it fails. 
\end{abstract}

\subjclass[2000]{16A39, 16K40, 16E10, 16W50}





\keywords{Hopf algebra, noetherian, Gelfand-Kirillov dimension}

\thanks{This research was partially supported by grants from the NSF
(USA) and by Leverhulme Research Interchange Grant F/00158/X (UK)}

\maketitle

{\small This paper is
dedicated to Susan Montgomery on the occasion of her 65th birthday.}


\setcounter{section}{-1}
\section{Introduction}

\bigskip

A study of general infinite dimensional noetherian Hopf algebras was 
initiated by Brown and the first-named author in 1997-98 \cite{BG1, Br2},
where  they summarized many nice properties of noetherian Hopf algebras,  
most of which are either PI (namely, satisfying a polynomial identity) 
or related to quantum groups, and posted  a list of interesting open 
questions. Some progress has been made, and  a few of the open questions 
have been solved since then. For example, the homological integral was 
introduced for Artin-Schelter Gorenstein noetherian Hopf algebras 
\cite{LWZ}, and it became a powerful tool in the classification of prime 
regular Hopf algebras of Gelfand-Kirillov dimension one \cite{BZ}. Some 
interesting examples were also discovered in \cite{BZ}. To further 
understand general infinite dimensional noetherian Hopf algebras, and to 
more easily reveal their structure, a large number of new examples would 
be extremely helpful. After the work \cite{BZ}, it is natural to 
consider Hopf algebras of Gelfand-Kirillov dimension two. 

In the present work, we aim to understand noetherian and/or affine Hopf
algebras $H$ under the basic assumption
\begin{enumerate}
\item[\Hzr]
{\sf $H$ is an integral domain with Gelfand-Kirillov dimension two, 
over an algebraically closed field $k$ of characteristic zero.}
\end{enumerate}
Our analysis requires an additional homological assumption, namely
\begin{enumerate}
\item[\nat]
{\sf $\Ext^1_H({}_Hk,{}_Hk) \ne 0$, where $_Hk$ denotes the trivial 
left $H$-module.}
\end{enumerate}
The condition \nat\ is equivalent to the condition that the 
corresponding quantum group contains a 
classical algebraic group of dimension 1 (see Theorem \ref{yythm3.9} 
for the algebraic version of this statement). The latter condition 
is analogous to  the fact that every quantum projective 2-space in 
the sense of Artin-Schelter \cite{ASc} contains a classical 
commutative curve of dimension 1.

The combination of \nat\ with \Hzr\ will be denoted \Hzrnat.
Our main result can be stated as follows.

\begin{theorem}
\label{yythm0.1} 
Let $H$ be a Hopf algebra satisfying \Hzrnat. Then $H$ is noetherian 
if and only if $H$ is affine, if and only if $H$ is isomorphic to 
one of the following:
\begin{enumerate}
\item[(I)]
The group algebra $k\Gamma$, where $\Gamma$ is either 
\begin{enumerate}
\item[(Ia)]
the free abelian group $\ZZ^2$, or 
\item[(Ib)]
the nontrivial semidirect product $\ZZ\rtimes\ZZ$.
\end{enumerate}
\item[(II)]
The enveloping algebra $U(\fg)$, where $\fg$ is either 
\begin{enumerate}
\item[(IIa)]
the $2$-dimensional abelian Lie algebra over $k$, or 
\item[(IIb)]
the Lie algebra over $k$ with basis $\{x,y\}$ and $[x,y]=y$.
\end{enumerate}
\item[(III)]
The Hopf algebras $A(n,q)$ from
Construction {\rm\ref{yycon1.1}}, for $n\ge0$.
\item[(IV)]
The Hopf algebras $B(n,p_0,\dots,p_s,q)$ from Construction 
{\rm\ref{yycon1.2}}.
\item[(V)] 
The Hopf algebras $C(n)$ from Construction
{\rm\ref{yycon1.4}}, for $n\ge2$.
\end{enumerate}
Aside from the cases $A(0,q)\cong A(0,q^{-1})$, the Hopf algebras 
listed above are pairwise non-isomorphic. 
\end{theorem}

The Hopf algebras $B(n,p_0,\dots,p_s,q)$ in Theorem \ref{yythm0.1}(IV)
provide a negative answer to \cite[Question 0.4]{WZ2} and 
\cite[Questions K and J]{Br3}; see Remark \ref{yyrem1.7}. 
Geometrically, Theorem \ref{yythm0.1} provides a list of 
2-dimensional quantum groups that are connected and satisfy the 
condition \nat.  

Following the above classification one can establish the
following common properties for these Hopf algebras. 

\begin{proposition}
\label{yyprop0.2}
Let $H$ be a Noetherian Hopf algebra satisfying \Hzrnat.
\begin{enumerate}
\item $\Kdim H=2$, and $\gldim H=2$ or $\infty$.
\item $H$ is Auslander-Gorenstein and GK-Cohen-Macaulay, with injective
dimension $2$.
\item $\Spec H$ has normal separation.
\item $H$ satisfies the strong second layer condition.
\item $\Spec H$ is catenary, and $\height(P)+\GKdim(H/P)=2$ for all prime
ideals $P$ of $H$.
\item $H$ satisfies the Dixmier-Moeglin equivalence.
\item
$H$ is a pointed Hopf algebra. 
\item
$\dim_k \Ext^1_H(_Hk,_Hk)=1$ if and only if $H$ is
not commutative. 
\end{enumerate}
\end{proposition}

We conjecture that Proposition \ref{yyprop0.2}(a,b,c,d,e,f) hold for 
affine and/or noetherian prime Hopf algebras of GK-dimension two.
Let $A$ be the group algebra given in \cite[Example 8.5]{LWZ}; then
the quotient Hopf algebra $A/[A,A]$ is isomorphic to the group 
algebra $k(\ZZ/2\ZZ)^{3}$ which is finite dimensional. By
Theorem \ref{yythm3.9}, $\Ext^1_A(k,k)=0$, and hence \nat\ fails 
for this prime affine noetherian Hopf algebra of GK-dimension two.

The condition \nat\ may follow from \Hzr\ as we have no 
counterexamples. It is only natural to conjecture \nat\ when 
$\GKdim H \le 2$, since examples abound in higher GK-dimensions, 
starting with $U({\mathfrak{sl}}_2(k))$. Completely different
ideas are needed to handle Hopf algebras satisfying \Hzr\ but
not \nat. We emphasize this key question: 

\begin{question}
\label{yyque0.3} Does \nat\ follow from \Hzr?
\end{question}

\begin{subsec} 
\label{yysub0.4}
Let $k$ be a commutative base field which is algebraically closed of 
characteristic $0$. All algebras, tensor products, linear maps, and 
algebraic groups in the paper are taken over $k$. Throughout, $H$ 
will denote a Hopf algebra over $k$, but we do not impose any of the 
hypotheses such as \Hzr\ at first. Our reference for basic material 
and notation about Hopf algebras is \cite{Mon}. In particular, we 
denote all counits, comultiplications, and antipodes by the symbols 
$\eps$, $\Delta$, and $S$, respectively, sometimes decorated by 
subscripts indicating the name of the Hopf algebra under consideration. 
We use the notations $U(\fg)$, $k\Gamma$, and $\cO(G)$ for enveloping 
algebras, group algebras, and coordinate rings of algebraic groups, 
respectively, taken with their standard Hopf algebra structures.
\end{subsec}

\section{Constructions of Hopf algebras satisfying \Hzrnat}
\label{yysec1}

\bigskip

In this section, we construct three families of Hopf algebras which,
together with two group algebras and two enveloping algebras, cover
the Hopf algebras classified in Theorem \ref{yythm0.1}. All of them are
generated (as algebras) by grouplike and skew primitive elements, and only
one family requires more than two such generators (counting a grouplike
element and its inverse as one generator).

Recall that a {\it grouplike element\/} in a Hopf algebra $H$ is an
element $g$ such that $\eps(g)=1$ and $\Delta(g)= g\otimes g$. It is not
necessary to specify $S(g)$, since that is forced by the antipode axiom:
the equations $\eps(g)1= S(g)g= gS(g)$ require that
$S(g)=g^{-1}$. The set of grouplike elements in $H$ forms
a subgroup of the group of units of $H$, denoted $G(H)$. A {\it skew
primitive element\/} in $H$ is an element
$p$ such that
$\Delta(p)= g\otimes p+ p\otimes h$ for some grouplike elements $g$ and
$h$. Here the value of
$\eps(p)$ is already forced by the counit axiom: since $p= 1p+\eps(p)h$
and $h\ne 0$, we must have $\eps(p)=0$. The antipode axiom yields
$0= g^{-1}p+S(p)h$, which requires $S(p)= -g^{-1}ph^{-1}$. If
desired, we can change to a situation in which either $g=1$ or $h=1$,
since $\Delta(g^{-1}p)= 1\otimes g^{-1}p+ g^{-1}p\otimes g^{-1}h$ and
$\Delta(ph^{-1})= gh^{-1}\otimes ph^{-1}+ ph^{-1}\otimes 1$. The term
{\it primitive element\/} is used when both $g=h=1$, that is, $\Delta(p)=
1\otimes p+ p\otimes 1$.

When defining a Hopf algebra structure on an algebra $A$ given by
generators and relations, it suffices to check the Hopf algebra axioms on
a set of algebra generators for $A$. This is obvious for the counit and
coassociativity axioms, which require that various algebra homomorphisms
coincide. As for the antipode axiom, it suffices to check it on
monomials in the algebra generators, and that follows from the case of a
single generator by induction on the length of a monomial.

\begin{construction} 
\label{yycon1.1}
Let $n\in\ZZ$ and $q\in\kx$, and set $A= k\langle x^{\pm1},y \mid
xy=qyx\rangle$.
\begin{enumerate}
\item There is a unique Hopf algebra structure on $A$ under which $x$ is
grouplike and $y$ is skew primitive, with $\Delta(y)= y\otimes1+
x^n\otimes y$. This Hopf algebra satisfies \Hzrnat, and we shall denote it
$A(n,q)$.
\item Let $m\in\ZZ$ and $r\in\kx$. Then $A(m,r)\cong A(n,q)$ if and only
if either $(m,r)= (n,q)$ or $(m,r)= (-n,q^{-1})$.
\end{enumerate}
\end{construction}

\begin{remark*} The Hopf algebra $A(1,q^2)$ is well known -- it is the
quantized enveloping algebra of the positive Borel subalgebra of
$\mathfrak{sl}_2(k)$ (e.g., \cite[I.3.1, I.3.4]{BG}). Moreover, $A(2,q)$
appears as the quantized Borel subalgebra of the variant
$\check{U}_{q^{-1}}(\mathfrak{sl}_2(k))$ \cite[\S3.1.2, p.~57]{KS}.

Because of (b) above, we may always assume that $n\ge0$ when
discussing $A(n,q)$.
\end{remark*}

\begin{proof} (a) It is clear that there exist unique $k$-algebra
homomorphisms
$\eps:A\rightarrow k$ and $\Delta: A\rightarrow A\otimes A$ and a unique
$k$-algebra anti-automorphism $S$ on $A$ such that
\begin{align}
\eps(x) &= 1  &\Delta(x) &= x\otimes x  &S(x) &= x^{-1}  \notag\\
\eps(y) &=0  &\Delta(y) &= y\otimes 1+ x^n\otimes y  &S(y) &= -x^{-n}y. 
\notag
\end{align}
The Hopf algebra axioms are easily checked on $x$ and $y$; we leave that
to the reader. That $A$ satisfies \Hzr\ is clear. Condition \nat\ 
follows from the fact that $_Hk$ is also an $(A/\langle y\rangle)$-module 
corresponding to $k[x^{\pm1}]/\langle x-1\rangle$, and the latter module 
has a non-split self-extension.

(b) Label the canonical generators of $A(m,r)$ as $u$, $u^{-1}$, $v$. If
$(m,r)= (-n,q^{-1})$, there is a Hopf algebra isomorphism $A(m,r)
\rightarrow A(n,q)$ sending $x\mapsto u^{-1}$ and $y\mapsto v$. 

Now assume we are given a Hopf algebra isomorphism $\phi: A(m,r)
\rightarrow A(n,q)$. Since $u$ generates the group of grouplikes in
$A(m,r)$, its image $\phi(u)$ must generate the group of grouplikes in
$A(n,q)$. As that group consists of the powers of $x$, the only
possibilities are $\phi(u)=x$ and $\phi(u)= x^{-1}$. By the first
paragraph, we may assume that $\phi(u)=x$. Let $f=\phi(v)$. Since 
$A(m,r)$ is generated by $u^{\pm1}$ and $v$ with relation $uv=r vu$, 
the algebra $A(n,q)$ is generated by $x^{\pm1}$ and $f$ with relation 
$xf=rfx$. This implies that $y=\sum_{i\geq 0} c_i(x)
f^i$ where $c_i(x)\in k[x^{\pm 1}]$. By counting the $y$-degree of
$\sum_{i\geq 0} c_i(x) f^i$, one sees that $f$ has $y$-degree 1 and 
$c_i(x)=0$ for all $i>1$. Writing $f=a(x)+b(x)y$ for $a(x),b(x)\in 
k[x^{\pm 1}]$ with $b(x)\ne 0$, we have
$$ra(x)x + rb(x)yx = rfx= xf= a(x)x+ qb(x)yx.$$
Since $b(x)\ne 0$, this equation forces $r=q$. We also have
$$y= c_0(x)+ c_1(x)\bigl( a(x)+b(x)y \bigr).$$
This implies that $b(x)=\alpha x^s$ and $c_1(x)=\alpha^{-1}x^{-s}$
for some $\alpha\in\kx$, $s\in\ZZ$. Applying $\phi$ to 
$\Delta(v)=v\otimes 1+u^m\otimes v$, we obtain $\Delta(f)=
f\otimes 1+ x^m\otimes f$, or
$$a(x\otimes x)+\alpha (x\otimes x)^s(y\otimes 1+x^n\otimes y)
=(a(x)+\alpha x^s y)\otimes 1 +x^m\otimes (a(x)+\alpha x^s y).$$ 
This equation forces $n=m$ and therefore $(n,q)= (m,r)$. 
\end{proof} 

Given an automorphism $\sigma$ of a $k$-algebra $A$, we write
$A[x^{\pm1};\sigma]$ for the corresponding skew Laurent polynomial ring,
with multiplication following the rule $xa=\sigma(a)x$ for $a\in A$.

\begin{construction} 
\label{yycon1.2}
Let $n,p_0,p_1,\dots,p_s$ be positive integers and
$q\in\kx$ with the following properties:
\begin{enumerate}
\item $s\ge2$ and $1<p_1< p_2< \cdots< p_s$;
\item $p_0\mid n$ and $p_0,p_1,\dots,p_s$ are pairwise relatively prime;
\item $q$ is a primitive $\ell$-th root of unity, where $\ell=
(n/p_0)p_1p_2\cdots p_s$.
\end{enumerate}

Set $m=p_1p_2\cdots p_s$ and $m_i= m/p_i$ for $i=1,\dots,s$. Choose an
indeterminate $y$, and consider the subalgebra $A= k[y_1,\dots,y_s]
\subset k[y]$ where $y_i= y^{m_i}$ for $i=1,\dots,s$. The $k$-algebra
automorphism of $k[y]$ sending $y\mapsto qy$ restricts to an automorphism
$\sigma$ of $A$. There is a unique Hopf algebra structure on
the skew Laurent polynomial ring $B= A[x^{\pm1};\sigma]$ such that $x$ is
grouplike and the $y_i$ are skew primitive, with $\Delta(y_i)=
y_i\otimes1 + x^{m_in}\otimes y_i$ for $i=1,\dots,s$.

The Hopf algebra $B$ satisfies \Hzrnat. We shall denote it 
$B(n,p_0,\dots,p_s,q)$.
\end{construction}

\begin{remark*}
The simplest case of Construction \ref{yycon1.2}
is when $n=p_0=1$, $p_1=2$ and $p_2=3$. The resulting Hopf
algebra $B(1,1,2,3,q)$ is new, as far as we are aware.
\end{remark*}

\begin{proof} Note that $m_ip_i= m = m_jp_j$ for all $i,j>0$. We claim
that $A$ can be presented as the commutative $k$-algebra with generators
$y_1,\dots,y_s$ and relations $y_i^{p_i}= y_j^{p_j}$ for $1\le i<j\le s$.
Consequently, $B$ can be presented as the $k$-algebra with generators
$x,x^{-1},y_1,\dots,y_s$ and the following relations:
\begin{equation} 
\label{E1.2.1} \tag{E1.2.1}
\begin{aligned}
xx^{-1} &= x^{-1}x =1  \\
xy_i &= q^{m_i}y_ix  &\qquad\qquad&(1\le i\le s)\\
y_iy_j &= y_jy_i  &&(1\le i<j\le s)  \\
y_i^{p_i} &= y_j^{p_j}  &&(1\le i<j\le s).  
\end{aligned}
\end{equation}

Let $R= k[x_1,\dots,x_s]$ be a polynomial ring in $s$ indeterminates, and
$I$ the ideal of $R$ generated by $x_i^{p_i}- x_j^{p_j}$ for $1\le i<j\le
s$. There is a surjective
$k$-algebra homomorphism $\phi: R/I \rightarrow A$ sending each $x_i+I
\mapsto y_i$, and we need to show that $\phi$ is an isomorphism. It
suffices to exhibit elements of $R/I$ which span it (as a vector space
over $k$) and which are mapped by $\phi$ to linearly independent elements
of $A$. To this end, use multi-index notation $x^d$ for monomials in $R$,
and write $\xbar^d= x^d+I$ for the corresponding cosets in $R/I$. Set
$$D= \{ d= (d_1,\dots,d_s)\in \Znonneg^s \mid d_i<p_i \text{\ for\ }
i=2,\dots,s\},$$ 
and observe that the elements $\xbar^d$ for $d\in D$
span $R/I$. The map
$\phi$ sends
$$x^d= x_1^{d_1}x_2^{d_2} \cdots x_s^{d_s} \longmapsto y_1^{d_1}y_2^{d_2}
\cdots y_s^{d_s}= y^{\mu\cdot d} \in k[y],$$
where $\mu= (m_1,\dots,m_s)$ and $\mu\cdot d$ denotes the inner product
$m_1d_1+ \cdots+ m_sd_s$. Thus, it suffices to show that the map
$d\mapsto \mu\cdot d$ from $D\rightarrow \Znonneg$ is injective.

If $d,d'\in D$ and $\mu\cdot d= \mu\cdot d'$, then 
$$m_1(d_1-d'_1) + \cdots+ m_s(d_s-d'_s) = 0.$$
Since $p_i\mid m_j$ for $i\ne j$ and $p_i$, $m_i$ are relatively prime,
$p_i \mid d_i-d'_i$ for all $i$. For $i\ge2$, it follows that $d_i=d'_i$,
given that $0\le d_i,d'_i <p_i$. This leaves $m_1(d_1-d'_1)=0$, whence
$d_1=d'_1$ and consequently $d=d'$, as desired. The claim is now
established.

It is clear that there is a $k$-algebra homomorphism $\eps: B\rightarrow
k$ such that $\eps(x)=1$ and $\eps(y_i)=0$ for all $i$. 

To construct a comultiplication map, we show that $x\otimes x$,
$x^{-1}\otimes x^{-1}$, and the elements $\delta_i= y_i\otimes1 +
x^{m_in}\otimes y_i$ in $B\otimes B$ satisfy the relations \eqref{E1.2.1}.
First, note that $(x\otimes x)\delta_i= q^{m_i}\delta_i(x\otimes x)$ for
all $i$. Next, since $m \mid m_im_j$ and hence $\ell\mid m_im_jn$ for
$i\ne j$, we have
$$x^{m_in}y_j= q^{m_im_jn}y_jx^{m_in}= y_jx^{m_in}$$
and similarly $x^{m_jn}y_i= y_ix^{m_jn}$ for $1\le i<j\le s$. It follows
that $\delta_1,\dots,\delta_s$ commute with each other. Finally, 
\begin{equation} 
\label{E1.2.2} \tag{E1.2.2}
\begin{aligned}
\delta_i^{p_i} &= \sum_{r=0}^{p_i} {p_i\choose r}_{q^{m_i^2n}}
y_i^{p_i-r}(x^{m_in})^r\otimes y_i^r  \\
 &= y_i^{p_i}\otimes1+
x^{m_inp_i}\otimes y_i^{p_i}= y_i^{p_i}\otimes1+
x^{mn}\otimes y_i^{p_i},
\end{aligned}
\end{equation}
making use of the fact that the $q^{m_i^2n}$-binomial coefficients
${p_i\choose r}_{q^{m_i^2n}}$ vanish for $0< r< p_i$, since $q^{m_i^2n}$
is a primitive $p_i$-th root of unity. To see the latter, observe that on
one hand,
$q^{m_i^2np_i} =1$ because $\ell$ divides $mn= m_inp_i$. On the other
hand, if $q^{m_i^2nt}=1$ for some positive integer $t$, then $\ell\mid
m_i^2nt$ and
$mn= p_0\ell$ divides $p_0m_i^2nt$, whence $p_1p_2\cdots p_s=m$ divides
$p_0t\prod_{j\ne i} p_j^2$. Since $p_0,\dots,p_s$ are pairwise relatively
prime, it follows that $p_i\mid t$, as desired. Equations \eqref{E1.2.2}
imply that $\delta_i^{p_i}= \delta_j^{p_j}$ for all $i$, $j$. 

We have now shown that $(x\otimes
x)^{\pm1},\delta_1,\dots,\delta_s$ satisfy \eqref{E1.2.1}.
Therefore, there is a $k$-algebra homomorphism $\Delta: B\rightarrow
B\otimes B$ such that $\Delta(x)= x\otimes x$ and $\Delta(y_i)= \delta_i$
for all $i$. Coassociativity and the counit axiom are easily seen to hold
for the generators $x,x^{-1},y_1,\dots,y_s$.

Finally, we construct a $k$-algebra anti-automorphism $S$ on
$B$ such that $S(x)= x^{-1}$ and $S(y_i)= w_i := -x^{-m_in}y_i$ for all
$i$. We need to show that $x^{-1},x,w_1,\dots,w_s$ satisfy the reverse of
\eqref{E1.2.1}, i.e., the corresponding relations with products reversed.
The first relation is trivially satisfied, and it is clear that
$w_ix^{-1}= q^{m_i}x^{-1}w_i$ for all $i$. Next,
$$w_iw_j= x^{-m_in}y_ix^{-m_jn}y_j= q^{m_im_jn}x^{-(m_i+m_j)n}y_iy_j=
x^{-(m_i+m_j)n}y_iy_j $$
for $i\ne j$, because $\ell\mid m_im_jn$ in that case. Hence, $w_iw_j=
w_jw_i$ for all $i$, $j$. For each $i$, we have
$$w_i^{p_i}= (-1)^{p_i}q^{a_i}x^{-m_inp_i}y_i^{p_i} \qquad\quad
\text{where} \qquad\quad a_i= m_i^2np_i(p_i-1)/2.$$
If $p_i$ is odd, then $\ell\mid a_i$, whence $(-1)^{p_i}q^{a_i}= -1$. If
$p_i$ is even, then $(-1)^{p_i}q^{a_i}= q^{-m_i^2n(p_i/2)}= -1$ because
$q^{m_i^2n}$ is a primitive $p_i$-th root of unity (as shown above).
Thus, $w_i^{p_i}= -x^{-mn}y_i^{p_i}$ for all $i$, and so $w_i^{p_i}=
w_j^{p_j}$ for all $i$, $j$. The relations needed for the existence of
$S$ are now established.

The antipode axiom is easily checked on the generators of $B$. Therefore
the algebra $B$, equipped with the given $\eps$, $\Delta$, $S$, is a Hopf
algebra. It is clear that $B$ satisfies \Hzr, and \nat\ follows from the 
fact that $H/\langle y_1,\dots,y_s\rangle \cong k[x^{\pm1}]$. 
\end{proof}

\begin{lemma} 
\label{yylem1.3}
Let $n,p_0,\dots,p_s\in \Zpos$, $q\in\kx$ and
$n',p'_0,\dots,p'_t\in\Zpos$, $r\in\kx$ satisfy conditions {\rm (a)--(c)}
of Construction {\rm\ref{yycon1.2}}. Then 
$$B(n,p_0,\dots,p_s,q) \cong B(n',p'_0,\dots,p'_t,r) \iff
(n,p_0,\dots,p_s,q)= (n',p'_0,\dots,p'_t,r).$$
\end{lemma}

\begin{proof} Set $B= B(n,p_0,\dots,p_s,q)$ and
$B'= B(n',p'_0,\dots,p'_t,r)$, and assume there is a Hopf algebra
isomorphism $\phi: B'\rightarrow B$. Label the generators of $B$ as
$x^{\pm1},y_1,\dots,y_s$ as in Construction \ref{yycon1.2}, and
correspondingly use $u^{\pm1},v_1,\dots,v_t$ for the generators of $B'$.
Since the groups of grouplike elements in $B$ and $B'$ are infinite
cyclic, generated by $x$ and $u$ respectively, we must have $\phi(u)=x$ or
$\phi(u)= x^{-1}$.

Since $\ell= (n/p_0)m_ip_i$ with $(n/p_0)p_i >1$, we have $q^{m_i} \ne
1$ for $i=1,\dots,s$. But $[x,y_i]= (q^{m_i}-1)y_ix$, so we obtain
$y_i\in [B,B]$. Therefore $[B,B]$ equals the ideal of $B$
generated by $y_1,\dots,y_s$. We then compute that the algebra $B_0$ of
coinvariants of $B$ for the right coaction $\rho: B\rightarrow B\otimes
\bigl( B/[B,B] \bigr)$ induced from $\Delta$ is precisely
$k[y_1,\dots,y_s]$. Likewise, the corresponding algebra $B'_0$ of
coinvariants in $B'$ must equal
$k[v_1,\dots,v_t]$. On the other hand, $B_0$ and $B'_0$ are defined
purely in terms of the Hopf algebra structures of $B$ and $B'$, so $\phi$
must map $B'_0$ isomorphically onto $B_0$. This isomorphism extends to
the integral closures of these domains, i.e., to an isomorphism
$\phihat: k[v]\rightarrow k[y]$. Observe that since $\phi$ maps 
$B'_0\cap \ker\eps$ to $B_0\cap \ker\eps$, we must have $\phihat(v)= 
\lambda y$ for some $\lambda\in \kx$. In particular, $\phihat$ is 
a graded isomorphism.

Now $m_1,\dots,m_s$ are the minimal generators of $\{j\in\Znonneg \mid
y^j\in B_0\}$, and similarly for $m'_1,\dots,m'_t$. Hence, these two sets
of integers must coincide, and since they are listed in descending order,
we conclude that $s=t$ and $m'_i=m_i$ for $i=1,\dots,s$. Recall that $m=
p_1p_2\cdots p_s$ is the least common multiple of the $m_i$, and
similarly for $m'= p'_1p'_2\cdots p'_s$. Consequently, $m'=m$, and so
$p'_i= m/m_i= p_i$ for $i=1,\dots,s$. 

Now $\phi$ sends each $v_i\mapsto \lambda_i y_i$ where $\lambda_i= 
\lambda^{m_i}$. In particular, since
$(\phi\otimes\phi)\Delta(v_1)= \Delta\phi(v_1)$, we obtain
$$\lambda_1 y_1\otimes 1 + x^{\pm m'_1n'}\otimes \lambda_1 y_1=
\lambda_1 y_1\otimes 1+ \lambda_1 x^{m_1n}\otimes y_1 \,,$$
whence $\phi(u)=x$ and $n'=n$. The relations $uv_i= r^{m'_i}v_iu$
now imply that $xy_i= r^{m'_i}y_ix$, whence $r^{m_i}= r^{m'_i}=
q^{m_i}$ for all $i$. In particular, since $r^{m_1}$ is a primitive
$\ell'/m_1$-th root of unity and $q^{m_1}$ is a primitive $\ell/m_1$-th
root of unity, we obtain $(n'/p'_0)p'_1= \ell'/m'_1= \ell/m_1=
(n/p_0)p_1$, and consequently $p'_0=p_0$. Finally, since
$\text{gcd}(m_1,\dots,m_s)=1$, there are integers $a_1,\dots,a_s$ such
that $a_1m_1+ \cdots+ a_sm_s=1$, whence
$$q= (q^{m_1})^{a_1} (q^{m_2})^{a_2} \cdots (q^{m_s})^{a_s}=
(r^{m_1})^{a_1} (r^{m_2})^{a_2} \cdots (r^{m_s})^{a_s}= r.$$
Therefore $(n,p_0,\dots,p_s,q)= (n',p'_0,\dots,p'_t,r)$.
\end{proof} 

Given a derivation $\delta$ on a $k$-algebra $A$, we write
$A[x;\delta]$ for the corresponding skew polynomial ring,
with multiplication following the rule $xa=ax+\delta(a)$ for $a\in A$.

\begin{construction} 
\label{yycon1.4}
Let $n$ be a positive integer, and set $C=k[y^{\pm1}] \bigl[ x;
(y^n-y)\frac{d}{dy} \bigr]$.
\begin{enumerate}
\item There is a unique Hopf
algebra structure on $C$ such that $y$ is grouplike and $x$ is
skew primitive, with $\Delta(x)= x\otimes y^{n-1}+ 1\otimes x$. The Hopf
algebra $C$ satisfies \Hzrnat, and we shall denote it $C(n)$.
\item For $m,n\in\Zpos$, the Hopf algebras $C(m)$ and $C(n)$
are isomorphic if and only if $m=n$. In fact, they are isomorphic as
rings only if $m=n$.
\end{enumerate}
\end{construction}

\begin{remark*}
Since $C(1) \cong A(1,1)$, we will usually assume that
$n\ge2$ when discussing $C(n)$. These Hopf algebras appear to be new, as
far as we are aware.
\end{remark*}

\begin{proof} (a) There is clearly a unique $k$-algebra
homomorphism $\eps: C\rightarrow k$ such that $\eps(x)=0$ and
$\eps(y)=1$. Since
$$[x\otimes y^{n-1}+ 1\otimes x,\, y\otimes y]= (y^n-y)\otimes y^n+
y\otimes (y^n-y)= (y\otimes y)^n- y\otimes y,$$
there is a unique $k$-algebra homomorphism $\Delta: C\rightarrow C\otimes
C$ such that $\Delta(y)=
y\otimes y$ and $\Delta(x)= x\otimes y^{n-1}+ 1\otimes x$. The relation
$xy-yx= y^n-y$ implies $xy^{-1}- y^{-1}x= y^{-1}-y^{n-2}$, whence 
$$y^{-1}(-xy^{1-n})- (-xy^{1-n})y^{-1}= (y^{-1}-y^{n-2})y^{1-n}=
y^{-n}-y^{-1}.$$
Consequently, there is a unique $k$-algebra anti-automorphism $S$ of $C$
such that $S(x)=
-xy^{1-n}$ and $S(y)= y^{-1}$.

It  is routine to check the Hopf algebra axioms for the generators $x$,
$y$, $y^{-1}$ of the algebra $C$, and it is clear that $C$
satisfies \Hzr. For \nat, observe that $C(y-1)$ is an ideal of $C$ and
$C/C(y-1)\cong k[x]$. 

(b) First note that $C(m)$ is commutative if and only if $m=1$, and 
similarly for $C(n)$. Hence, we need only consider the cases when 
$m,n\ge2$. Under that assumption, we claim that the quotients of 
$C(m)$ and $C(n)$ modulo their commutator ideals have Goldie 
ranks $m-1$ and $n-1$, respectively. Statement (b) then follows. 

Fix an integer $n\ge2$ and $H=C(n)$, and note that $[H,H]= (y^n-y)H$. Set
$\partial= (y^n-y)\frac{d}{dy}$, and let $\xi_1,\dots,\xi_{n-1}$ be the
distinct $(n-1)$st roots of unity in $k$.
For $i=1,\dots,n-1$, the ideal
$(y-\xi_i)k[y^{\pm1}]$ is a $\partial$-ideal of $k[y^{\pm1}]$, whence
$(y-\xi_i)H$ is an ideal of $H$, and $H/(y-\xi_i)H \cong k[x]$. Since the
$y-\xi_i$ generate pairwise relatively prime ideals of $k[y^{\pm1}]$
whose intersection equals the ideal generated by $y^n-y$, the
corresponding statements hold for ideals of $H$. Consequently,
$H/(y^n-y)H \cong \prod_{i=1}^{n-1} H/(y-\xi_i)H$, a direct product of
$n-1$ copies of $k[x]$. 
Therefore $H/(y^n-y)H$ has Goldie rank $n-1$, as claimed. 
\end{proof}

\begin{remark} 
\label{yyrem1.5} 
K.A. Brown has pointed out that some of the Hopf algebras constructed 
here are lifts of others in the sense of Andruskiewitsch and Schneider 
(e.g., \cite{AnS}).
First, the Hopf algebra $C(n)$ of Construction 
\ref{yycon1.4} is a lift of the Hopf algebra $A(n-1,1)$ of Construction 
\ref{yycon1.1} -- the associated graded ring of $C(n)$ with respect to 
its coradical filtration is isomorphic (as a graded Hopf algebra) to 
$A(n-1,1)$. This raises the question of lifting the $A(-,-)$ in general, 
which is easily done. For instance, given $n\in\ZZ$ and $q\in\kx$, let 
$C(n,q)$ be the $k$-algebra given by generators $y^{\pm1}$ and $x$ and 
the relation $xy= qyx+y^n-y$. It is easy to check that there is a unique 
Hopf algebra structure on $C(n,q)$ such that $y$ is grouplike and $x$ is
skew primitive, with $\Delta(x)= x\otimes y^{n-1}+ 1\otimes x$. The 
associated graded ring of this Hopf algebra, with respect to its 
coradical filtration, is isomorphic to $A(n-1,q^{-1})$. However, 
$C(n,q)$ is not new -- it turns out that $C(n,q) \cong A(n-1,q^{-1})$ 
when $q\ne 1$.
\end{remark}

\begin{proposition} 
\label{yyprop1.6}
The Hopf algebras listed in Theorem {\rm\ref{yythm0.1}(I)--(V)} are
pairwise non-isomorphic, aside from the cases $A(0,q)\cong 
A(0,q^{-1})$ for $q\in\kx$. 
\end{proposition}

\begin{proof} The commutative Hopf algebras in this list are the group
algebra $k[\ZZ^2]= k[x^{\pm1},y^{\pm1}]$, the enveloping algebra
$U(\gfrak)= k[x,y]$, and the $A(n,1)= k[x^{\pm1},y]$ for $n\ge0$. These
three types are pairwise non-isomorphic because their groups of grouplike
elements are different -- they are free abelian of ranks $2$, $0$, and
$1$, respectively. Further, $A(m,1)\not\cong A(n,1)$ for distinct
$m,n\in\Znonneg$ by Construction \ref{yycon1.1}(b).

The noncommutative, cocommutative Hopf algebras in the list are the group
algebra $k[\ZZ\rtimes\ZZ]$, the enveloping algebra $U(\gfrak)$ with
$\gfrak$ non-abelian, and the $A(0,q)$  with $q\ne 1$. The three types are
again pairwise non-isomorphic because their groups of grouplike elements
are different -- namely $\ZZ\rtimes\ZZ$, the trivial group, and the
infinite cyclic group. Moreover, $A(0,q) \cong A(0,r)$ (for $q,r\in\kx$)
 if and only if $q= r^{\pm1}$ by Construction \ref{yycon1.1}(b). 

Each $B(-)= B(n,p_0,\dots,p_s,q)= k[y_1,\dots,y_s][x^{\pm1};\sigma]$ as
algebras. Since $s\ge2$, the subalgebra $k[y_1,\dots,y_s]$ has infinite
global dimension, from which it follows that $B(-)$ has
infinite global dimension (e.g., \cite[Theorem 7.5.3(ii)]{MR}). However,
all the other Hopf algebras in the list have global dimension $2$.
Therefore no
$B(-)$ is isomorphic to any of the others. Moreover, the
$B(-)$ themselves are pairwise non-isomorphic by Lemma
\ref{yylem1.3}.

The abelianization $A(n,q)/[A(n,q),A(n,q)]$ is isomorphic to
either $k[x^{\pm1}]$ (when $q\neq 1$) or $k[x^{\pm1},y]$ (when
$q=1$). As shown in the proof of Construction \ref{yycon1.4}, the 
abelianization $C(m)/[C(m),C(m)]$, when $m\ge2$, is isomorphic to a 
direct product of $m-1$ copies of $k[x]$. Therefore $C(m)$, when 
$m\ge2$, is not isomorphic to any $A(n,q)$.  Taking account of 
Constructions \ref{yycon1.1}(b) and \ref{yycon1.4}(b), the proposition 
is proved. 
\end{proof}

\begin{remark} 
\label{yyrem1.7}
In \cite[Question K]{Br3}, Brown asked whether a noetherian
Hopf algebra which is a domain over a field of characteristic zero must
have finite global dimension. A similar question was asked by 
Wu and the second-named author in 
\cite[Question 0.4]{WZ2}. The Hopf algebras $B(n,p_0,\dots,p_s,q)$
provide a negative answer to this question, since they have infinite
global dimension as shown in the proof of
Proposition \ref{yyprop1.6}. The Hopf algebras $B(n,p_0,\dots,p_s,q)$
also provide a negative answer to \cite[Question J]{Br3}. 
\end{remark}

\section{Commutative Hopf domains of GK-dimensions $1$ and $2$}
\label{yysec2}

While our main interest is in Hopf algebras of GK-dimension $2$, these 
may have subalgebras and/or quotients of GK-dimension $1$, and it is 
helpful to have a good picture of that situation. Thus, we begin by 
classifying commutative Hopf algebras which are domains of 
GK-dimension $1$ (= transcendence degree $1$) over $k$. 

Recall that the antipode of any commutative (or cocommutative) Hopf 
algebra $H$ satisfies $S^2= \id$ \cite[Corollary 1.5.12]{Mon}. From 
this, we can see that $H$ is the directed union of its affine Hopf
subalgebras, as follows. Any finite subset $X\subseteq H$ must be 
contained in a finite dimensional subcoalgebra, say $C$, of 
$H$ \cite[Theorem 5.1.1]{Mon}. Since $S$ is an anti-coalgebra morphism, 
$S(C)$ is a subcoalgebra of $H$, and hence so is $D:= C+S(C)$. 
Moreover, $S(D)\subseteq D$. The subalgebra $A$ of $H$ generated by 
$D$ is thus a Hopf subalgebra containing $X$. Choose a
basis $Y$ for $C$. Then $A$ is generated by $Y\cup S(Y)$, proving that $A$
is affine.

\begin{proposition}  
\label{yyprop2.1}
Assume that the Hopf algebra $H$ is a commutative domain with 
$\GKdim H = 1$. Then $H$ is isomorphic to one of the following:
\begin{enumerate}
\item An enveloping algebra $U(\fg)$, where $\dim_k \fg =1$.
\item A group algebra $k\Gamma$, where $\Gamma$ is infinite cyclic.
\item A group algebra $k\Gamma$, where $\Gamma$ is a non-cyclic 
torsionfree abelian group of rank $1$, i.e., a non-cyclic subgroup 
of ${\mathbb Q}$.
\end{enumerate}
\end{proposition}

\begin{proof} If $H$ is affine, then $H\cong \cO(G)$ for some connected
$1$-dimensional algebraic group $G$ over $k$. There are only two
possibilities: either $G\cong \kx$ or $G\cong k^+$ \cite[Theorem
III.10.9]{Bo}. These correspond to cases (b) and (a), respectively.

Now suppose that $H$ is not affine, and view $H$ as the directed 
union of its affine Hopf subalgebras. 

Since $H$ is a domain, its only finite dimensional subalgebra is $k1$.
Hence, any nontrivial affine Hopf subalgebra of $H$ has GK-dimension
$1$, and so must have one of the forms (a) or (b). Among these Hopf 
algebras, there are no proper embeddings of the type 
$U(\fg)\hookrightarrow U(\fg')$, because the only primitive elements 
of $U(\fg')$ are the elements of $\fg'$, and there are no embeddings 
of the type $k\Gamma \hookrightarrow U(\fg)$, because the only units 
of $U(\fg)$ are scalars. Thus, $H$ must be a directed union
$\bigcup_{i\in I} H_i$ of Hopf subalgebras $H_i \cong k\Gamma_i$, 
where the $\Gamma_i$ are infinite cyclic groups.

For each index $i$, the group $G(H_i)$ is isomorphic to $\Gamma_i$, 
and is a $k$-basis for $H_i$. The group $\Gamma= G(H)$ is the directed 
union of the $G(H_i)$, so it is torsionfree abelian of rank $1$, and 
it is a $k$-basis for $H$. Therefore $H=k\Gamma$. We conclude by 
noting that $\Gamma$ is not cyclic, since $H$ is not affine.
\end{proof} 

Turning to the case of GK-dimension $2$, we need to describe the
$2$-dimensional connected algebraic groups over $k$. While the result 
is undoubtedly known, we have not located it in the literature, and 
so we sketch a proof.

\begin{lemma}  
\label{yylem2.2} 
Let $G$ be a $2$-dimensional connected algebraic group over $k$. 
Then $G$ is isomorphic {\rm(}as an algebraic group{\rm)} to either
\begin{enumerate}
\item 
One of the abelian groups $(\kx)^2$, $(k^+)^2$, $k^+\times \kx$, or
\item 
One of the semidirect products $k^+\rtimes \kx$ where $\kx$
acts on $k^+$ by $b.a= b^na$, for some positive integer $n$.
\end{enumerate}
\end{lemma}

\begin{remark*}
Note that the three groups listed in (a) are pairwise non-isomorphic. 
The group described in (b) can be written as the set $k^+\times\kx$ 
equipped with the product $(a,b)\cdot(a',b')= (a+b^na',bb')$. The 
center consists of the elements $(0,b)$ with $b^n=1$, so the center 
is cyclic of order $n$. Thus, none of these groups is abelian, and 
no two (for different $n$) are isomorphic. It is also possible to 
form a semidirect product with the multiplication rule 
$(a,b)\cdot(a',b')= (a+b^{-n}a',bb')$, but this group is isomorphic
to the previous one, via $(a,b) \mapsto (a,b^{-1})$.
\end{remark*}

\begin{proof}[Proof of Lemma \ref{yylem2.2}]
According to \cite[Corollary IV.11.6]{Bo}, 2-dimensional groups are 
solvable. Then, by \cite[Theorem III.10.6]{Bo}, $G= T\ltimes G_u$ where 
$T$ is a maximal torus and $G_u$ is the unipotent radical of $G$.

(i) If $\dim T =2$, then $G=T$ is a torus, isomorphic to $(\kx)^2$.

(ii) If $\dim T =0$, then $G=G_u$ is unipotent. Another part of 
\cite[Theorem III.10.6]{Bo} and its proof shows that $G$ has a closed 
connected normal subgroup $N$, contained in $Z(G)$, such that $N$ 
and $G/N$ are 1-dimensional. Being unipotent, $N\cong G/N\cong k^+$. 
Choose $x\in G\setminus N$ and let $H$ be the subgroup of $G$ 
generated by $N\cup\{x\}$. Then $H$ is abelian and $H/N$ is infinite. 
The latter implies $H/N$ is dense in $G/N$, so $H$ is dense in $G$. 
Thus, $G$ is abelian. Unipotence then implies $G\cong (k^+)^2$.

(iii) If $\dim T= 1$, then also $\dim G_u=1$, so $T\cong \kx$ and
$G_u\cong k^+$ (since $G_u$ is connected \cite[Theorem III.10.6]{Bo}). 
Hence, $G\cong k^+
\rtimes_\phi \kx$, for some homomorphism $\phi$ from $\kx$ to the 
group $\Aut(k^+)$ of algebraic group automorphisms of $k^+$. But
$\Aut(k^+) \cong \kx$, where $b\in \kx$ corresponds to the 
automorphism $a\mapsto ba$. The homomorphisms $\phi$ thus amount 
to the (rational) characters of $\kx$, namely the maps
$\kx\rightarrow\kx$ given by $b\mapsto b^n$ for $n\in\ZZ$. The 
case $n=0$ corresponds to the trivial action of $\kx$ on $k^+$, in 
which case $G\cong k^+\times\kx$. The cases $n>0$ give the groups 
listed in (b), and the cases $n<0$ give these same groups again.
\end{proof}

\begin{proposition}  
\label{yyprop2.3}
Assume that the Hopf algebra $H$ is a commutative domain with 
$\GKdim H =2$. Then $H$ is noetherian if and only if $H$ is affine, 
if and only if $H$ is isomorphic to one of the following:
\begin{enumerate}
\item 
An enveloping algebra $U(\fg)$, where $\fg$ is $2$-dimensional
abelian.
\item 
A group algebra $k\Gamma$, where $\Gamma$ is free abelian of rank $2$.
\item 
A Hopf algebra $A(n,1)$, for some nonnegative integer $n$.
\end{enumerate}
\end{proposition}

\begin{proof} If $H$ is affine, then because it is commutative, it 
is noetherian. The converse is Molnar's theorem \cite{Mol}. 

The Hopf algebras listed in (a)--(c) are clearly affine. Conversely, if $H$
is affine, then $H\cong \cO(G)$ for some $2$-dimensional connected
algebraic group $G$. The structure of $H$ follows from the descriptions of
$G$ given in Lemma \ref{yylem2.2}. Namely, the cases $G\cong (\kx)^2$,
$(k^+)^2$, or $k^+\times \kx$ give $H\cong k\Gamma$, $U(\fg)$, or $A(0,1)$,
respectively, while case (b) of Lemma \ref{yylem2.2} corresponds to 
$H\cong A(n,1)$ with $n>0$. 
\end{proof}

\section{$\Ext^1_H(k,k)$ and condition \nat}
\label{yysec3}

We discuss the relevance of $\Ext^1_H(k,k)$ (actually, just its
dimension) to the structure of general Hopf algebras $H$ over $k$. 
To simplify notation, we have written $k$ rather than $_Hk$
for the trivial left $H$-module, and we define
$$e(H) := \dim_k \Ext^1_H(k,k).$$
It follows from Lemma \ref{yylem3.1}(a) below that $e(H)$ also equals
$\dim_k \Ext^1_H(k_H,k_H)$.

If $H$ is an affine commutative Hopf algebra ${\mathcal O}(G)$,
then $e(H)$ equals the dimension of $G$.

For purposes of the present paper, the key result of this section is that
\Hzrnat\ together with affineness or noetherianness implies that $H$ has a
Hopf quotient of the form $k[t^{\pm1}]$ or $k[t]$ (Theorem \ref{yythm3.9}).

\begin{lemma}  
\label{yylem3.1}
Let $\fm= \ker\epsilon$ where $\epsilon$ is the counit of $H$. 
\begin{enumerate}
\item 
$e(H)= \dim_k (\fm/\fm^2)^*$.
\item 
If $H$ is either noetherian or affine, then $e(H) <\infty$.
\end{enumerate}
\end{lemma}

\begin{proof} (a) Since $\fm k=0$, the restriction map $\Hom_H(H,k)
\rightarrow \Hom_H(\fm,k)$ is zero. Inspecting the long exact sequence 
for $\Ext^\bullet_H(-,k)$ relative to the short exact sequence 
$0\rightarrow \fm\rightarrow H\rightarrow k\rightarrow 0$, we see that 
the connecting homomorphism $\Hom_H(\fm,k) \rightarrow \Ext^1_H(k,k)$ 
is an isomorphism. On the other hand, all maps in $\Hom_H(\fm,k)$ 
factor through $\fm/\fm^2$, so that $\Hom_H(\fm,k)\cong 
\Hom_H(\fm/\fm^2,k)$. Thus, $\Ext^1_H(k,k) \cong (\fm/\fm^2)^*$, and 
part (a) follows.

(b) If $H$ is noetherian, then $\fm/\fm^2$ is a finitely generated 
left module over $H$, and hence also over $H/\fm =k$, whence 
$\dim_k \fm/\fm^2 <\infty$, and so $e(H)<\infty$ by (a). If $H$ is 
generated as a $k$-algebra by elements $h_1,\dots,h_n$, write each 
$h_i= \alpha_i+m_i$ for some $\alpha_i\in k$ and $m_i\in \fm$, and 
note that $H$ is also generated as a $k$-algebra by $m_1,\dots,m_n$. 
It follows that $\fm/\fm^2$ is spanned by the cosets $m_i+\fm^2$. 
Again, $\dim_k \fm/\fm^2 <\infty$ and $e(H)<\infty$. 
\end{proof}

Further information will be obtained by working with the associated 
graded ring of $H$ relative to the descending filtration given by 
powers of $\fm$. Since $\fm$ is a Hopf ideal of $H$ (because $\epsilon$ 
is a Hopf algebra homomorphism), this associated graded ring will be a 
Hopf algebra, as follows. This result is presumably well known, but 
since we could not locate it in the literature, we sketch a proof. 

Recall that a \emph{positively graded Hopf algebra} is a Hopf algebra 
$B$ which is also a positively graded $k$-algebra of the form $B=
\bigoplus_{i\ge0} B_i$, such that $\Delta_B$, $\epsilon_B$, and $S_B$ 
are graded maps of degree zero. (We assume that $B\otimes B$ is 
positively graded by total degree, that is, $(B\otimes B)_i= 
\sum_{l+m=i} B_l\otimes B_m$ for all $i\ge0$.)

\begin{lemma}  
\label{yylem3.2} 
Let $K$ be a Hopf ideal of $H$, and $B=\bigoplus_{i\geq 0} K^i/K^{i+1}$ 
the corresponding associated graded ring {\rm(}where $K^0$ means 
$H${\rm)}. Set $L= H\otimes K+K\otimes H$, 
and let $C= \bigoplus_{n\ge0} L^n/L^{n+1}$ be the corresponding 
associated graded ring of $H\otimes H$.
\begin{enumerate}
\item There is a graded algebra isomorphism $\theta: B\otimes B\rightarrow
C$ which sends 
$$(a+K^{i+1})\otimes (b+K^{j+1}) \mapsto (a\otimes b)+L^{i+j+1}$$
for all $a\in K^i$ and $b\in K^j$, $i,j\ge0$.
\item The algebra $B$ becomes a positively graded Hopf algebra with
\begin{itemize}
\item $\Delta_B= \theta^{-1}\Deltabar$, where $\Deltabar: B\rightarrow C$
is the graded algebra homomorphism induced by $\Delta$;
\item $\epsilon_B= \epsilon_{B_0}\pi$, where $\pi: B\rightarrow B_0$ is the
canonical projection;
\item $S_B$ is the graded algebra anti-endomorphism of $B$ induced by $S_H$.
\end{itemize}
\end{enumerate}
\end{lemma}

\begin{proof} (a) Observe that the equalities $\bigoplus_{i=0}^n K^i\otimes
K^{n-i} = L^n$ induce linear isomorphisms $(B\otimes B)_n \rightarrow
L^n/L^{n+1}$ for all $n\ge0$. The direct sum of these isomorphisms gives a
graded linear isomorphism $\theta: B\otimes B\rightarrow C$ which acts as
described, and one checks that $\theta$ is an algebra homomorphism.

(b) Since $K$ is a Hopf ideal
of $H$, the coproduct $\Delta_H$ maps $K$ into $L$, and hence
induces a graded algebra homomorphism $\Deltabar: B\rightarrow C$. Composing
this map with $\theta^{-1}$, we obtain a graded algebra
homomorphism $\Delta_B: B \rightarrow B\otimes B$. Coassociativity of
$\Delta_B$ follows easily from coassociativity of $\Delta_H$.

It is clear that $\epsilon_B$ is a graded algebra homomorphism from $B$ to
$k$, where $k$ is concentrated in degree $0$. The counit
axiom for $B$ follows directly from that for $H$.
Finally, since $S_H$ is an algebra anti-endomorphism of $H$ mapping
$K\rightarrow K$, it induces a graded algebra anti-endomorphism $S_B$ of
$B$. All that remains is to check the antipode axiom.
\end{proof}

\begin{lemma}  
\label{yylem3.3}
Let $\fm= \ker\epsilon$, and write $\gr H$
for the graded algebra $\bigoplus_{i=0}^\infty 
\fm^i/\fm^{i+1}$.
\begin{enumerate}
\item With the operations described in Lemma {\rm \ref{yylem3.2}}, $\gr H$
is a positively graded Hopf algebra.
\item As a $k$-algebra, $\gr H$ is connected graded and generated in degree
$1$.
\item $\gr H$ is also connected as a coalgebra.
\item All homogeneous elements of degree $1$ in $\gr H$ are primitive
elements.
\item $\gr H$ is cocommutative.
\end{enumerate}
\end{lemma}

\begin{proof} Let $B := \gr H$ be the graded algebra defined in  
Lemma \ref{yylem3.2} with $K=\fm$.

(a) See Lemma \ref{yylem3.2}.

(b) It is clear that the algebra $B$ is generated in degree
$1$, and it is connected graded because $B_0=H/\fm=k$.

(d) Let $x\in B_1$. Since $B_0= k$, Lemma \ref{yylem3.2}
shows that we can write $\Delta_B(x)= 1\otimes y+ z\otimes 1$ for some
$y,z\in B_1$. Applying the counit axiom together with the fact that
$\epsilon_B(y)= \epsilon_B(z)=0$, we find that $y=z=x$. This shows that $x$
is a primitive element.

(e) Since $B$ is generated by $B_1$, cocommutativity follows from
(d).

(c) Since $B$ is cocommutative and $k$ is algebraically closed, $B$ is
pointed \cite[p. 76]{Mon}. Thus, any simple subcoalgebra of $B$ has the
form $kg$ for a grouplike element $g$. To prove that $B$ is a connected
coalgebra, it remains to show that $g=1$.
Write $g= x_0+\cdots+x_n$ for some $x_i\in B_i$, with $x_n\ne 0$. Then
$$\sum_{i=0}^n \Delta_B(x_i)= \Delta_B(g)= g\otimes g= \sum_{l,m=0}^n
x_l\otimes x_m \,,$$
with $x_n\otimes x_n\ne 0$. Since the homogeneous components of
$\Delta_B(g)$ have degree at most $n$, we can have $x_n\otimes x_n\ne 0$
only when $n=0$. Hence, $g\in k1$, and so $g=1$ because $\epsilon(g)=1$.
This proves that $B$ is indeed connected.
\end{proof}

In the following result, the dimension of an infinite dimensional vector
space $V$ is considered to be the symbol $\infty$, rather than an infinite
cardinal, so that $\dim_k V^*= \dim_k V$.

\begin{proposition}  
\label{yyprop3.4}
Let $\fm= \ker\epsilon$, and write $\gr H$ for
the positively graded Hopf algebra $\bigoplus_{i=0}^\infty 
\fm^i/\fm^{i+1}$.
\begin{enumerate}
\item $\gr H\cong U(\fg)$ for some positively graded Lie algebra $\fg$
which is generated in degree $1$. 
\item $e(H)= \dim_k (\gr H)_1 =\dim_k \fg_1 \leq \dim_k \fg
\leq \GKdim H$.
\end{enumerate}
\end{proposition}

\begin{proof} (a) By Lemma \ref{yylem3.3}, $B := \gr H$ is a connected
cocommutative Hopf algebra, generated in degree $1$ as an algebra. The
Cartier-Kostant Theorem (e.g., \cite[Theorem 5.6.5]{Mon}) thus implies 
that $B \cong U(\fg)$, where $\fg$ is the Lie algebra of primitive 
elements of $B$. 

We next claim that $\fg$ is a positively graded Lie algebra, $\fg=
\bigoplus_{i\ge1} \fg_i$, where $\fg_i= \fg\cap B_i$ for $i\ge1$. To
see this, consider $x\in\fg$, and write $x= x_0+\cdots+x_n$ with $x_i\in
B_i$ for $i=0,\dots,n$. Then
$$ \sum_{i=0}^n \Delta(x_i)= \Delta(x)= x\otimes1+ 1\otimes x= \sum_{i=0}^n
(x_i\otimes1 + 1\otimes x_i).$$
Since each $\Delta(x_i) \in (B\otimes B)_i$, we see that
$\Delta(x_i)= x_i\otimes1+ 1\otimes x_i$. Thus, $x_i\in\fg$ for all $i$.
Moreover, $B_0=k$ contains no nonzero primitive elements, so $x_0=0$.
Thus, $x$ lies in $\sum_{i\ge1} \fg_i$, establishing the claim.

If $\fh$ is the Lie subalgebra of $\fg$ generated by $\fg_1$, then 
$$B_1= \fg_1 \subseteq \fh\subseteq U(\fh)\subseteq U(\fg)= B.$$
Since $B$ is generated by $B_1$, we have $U(\fh)= U(\fg)$, whence
$\fh=\fg$. Therefore $\fg$ is generated in degree $1$.

(b) In view of Lemmas \ref{yylem3.1}(a) and \ref{yylem3.3}(d), we have
$$e(H)= \dim_k \fm/\fm^2 =\dim_k B_1 =\dim_k \fg_1\leq
\dim_k \fg.$$

To prove that $\dim_k \fg\leq \GKdim H$, we may assume that $\GKdim
H=d<\infty$.  Suppose on the contrary that $\dim_k \fg> d$. Then we
pick a  finite dimensional subspace $W\subseteq \fg_1= \fm/\fm^2$ such that 
$\GKdim U(L(W)) =\dim_k L(W)> d$, where $L(W)$ is the 
Lie subalgebra of $\fg$ generated by $W$. Lift $W$ to a finite dimensional
vector subspace $V\subseteq \fm$ such that $V\cap \fm^2 =0$
and $(V+\fm^2)/\fm^2=W$. Then, for arbitrarily small $b>0$, there
is an $a>0$ such that $\dim_k (k+V)^{n}\leq a n^{d+b}$ for all $n>0$.
Note that 
\begin{align*}
\dim_k (k+V)^n &=\dim_k \sum_{i=0}^n V^i
\geq \sum_{i=0}^n \dim_k (V^i+\fm^{i+1})/\fm^{i+1}  \\
 &=\sum_{i=0}^n \dim_k W^i =\dim_k (k+W)^n \geq a_1 n^{\dim_k L(W)},
\end{align*}
for some $a_1>0$ and all $n\ge0$, since $\dim_k L(W)=\GKdim U(L(W))$. Thus
we have
$$a n^{d+b} \ge a_1 n^{\dim_k L(W)}$$
for all $n>0$. Then $d+b\geq \dim_k L(W)$ and hence $d\geq\dim_k L(W)$
since $b$ can be arbitrarily small.  

We now have $\dim_k W \le \dim_k L(W) \le d$ for any finite dimensional
subspace $W\subseteq \fg_1$ such that $\dim_k L(W) >d$. Since the result
holds for all finite dimensional subspaces of $\fg_1$ containing such a
$W$, we must have $\dim_k \fg_1 \le d$. Taking the case $W= \fg_1$, we
obtain $\dim_k \fg \le d$, contradicting our assumption. Therefore $\dim_k
\fg\leq \GKdim H$.
\end{proof}

\begin{lemma}  
\label{yylem3.5} 
Let $\fm= \ker\epsilon$ and $J_{\fm}=\bigcap_{i\geq 0} \fm^i$, and write
$\gr H$ for the graded Hopf algebra $\bigoplus_{i=0}^\infty 
\fm^i/\fm^{i+1}$. If $\gr H$
is commutative, then $H/J_\fm$ is commutative.
\end{lemma}

\begin{proof} Proposition \ref{yyprop3.4} shows that 
$B:= \gr H \cong U(\fg)$ for a positively graded Lie
algebra $\fg$ generated in degree $1$. By hypothesis, $U(\fg)$ is
commutative, whence $\fg$ is abelian and $\fg=\fg_1$.

To prove that $H/J_\fm$ is commutative, it suffices to show that $H/\fm^n$
is commutative for all $n\ge0$. Note that $H/\fm^2=
k1\oplus \fm/\fm^2$ is automatically commutative. 

Now assume that $H/\fm^n$ is commutative for some $n\ge2$. Adopt the
notation of Lemma \ref{yylem3.2}, with $K=\fm$. Since
$$(H\otimes H)/(H\otimes\fm^n+ \fm^n\otimes H) \cong (H/\fm^n)\otimes
(H/\fm^n)$$
is commutative, so is $(H\otimes H)/L^n$. In particular, it follows
that $[L,L] \subseteq L^n$. An easy induction then yields $[L^i,L^j]
\subseteq L^{i+j+n-2}$ for all $i,j\ge1$.

Consider any elements $y_1,y_2 \in \fm$, and set $x_i= y_i+\fm^2 \in B_1=
\fg$ for $i=1,2$. Since the $x_i$ are primitive elements, we see that each
$\Delta(y_i)= y_i\otimes 1+ 1\otimes y_i+ u_i$ for some $u_i\in L^2$.
From the previous paragraph, $[u_i,L] \subseteq L^{n+1}$ for $i=1,2$, 
and $[u_1,u_2] \in L^{n+2}$. It follows that
$$\Delta([y_1,y_2])= [\Delta(y_1),\Delta(y_2)] \equiv [y_1,y_2]\otimes 1 +
1\otimes [y_1,y_2] \pmod{L^{n+1}},$$
and so $z:= [y_1,y_2]+\fm^{n+1}$ is a primitive element in $B_n$. As $B=
\gr H \cong U(\fg)$ with $\fg$ in degree $1$, all primitive elements of $B$
are in degree $1$. Consequently, $z=0$ and $[y_1,y_2] \in \fm^{n+1}$. Since
$\fm$ generates $H$, this proves that $H/\fm^{n+1}$ is commutative,
establishing the induction step. This finishes the proof. 
\end{proof}

\begin{proposition}  
\label{yyprop3.6} 
Let $H$ be either noetherian or affine. Set $\fm=
\ker\epsilon$, and write
$\gr H$ for the graded Hopf algebra $\bigoplus_{i=0}^\infty 
\fm^i/\fm^{i+1}$. Suppose that $\bigcap_{i\geq
0} \fm^i=0$. Then the following are equivalent:
\begin{enumerate}
\item
$H$ is commutative.
\item
$e(H)=\GKdim H$. 
\item
$\gr H$ is commutative.
\end{enumerate}
\end{proposition}

\begin{proof} By Lemma \ref{yylem3.1}, $e(H)= \dim_k \fm/\fm^2 <\infty$.

(a)$\Longrightarrow$(b): Assuming $H$ is commutative, it is
affine either by assumption or by Molnar's theorem \cite{Mol}. Hence,
$H\cong \cO(G)$ for some algebraic group $G$. Since algebraic groups are
smooth and homogeneous as varieties \cite[Proposition I.1.2]{Bo},
$H$ is regular and equidimensional. Consequently,
$$\GKdim H= \dim G= \gldim H= \dim_k \fm/\fm^2= e(H).$$

(b)$\Longrightarrow$(c): With $\fg$ as in Proposition
\ref{yyprop3.4}, condition (b) implies that $\dim_k\fg= \dim_k\fg_1$.
This implies that $\fg$ is abelian. Therefore $\gr H\cong U(\fg)$ is
commutative.

(c)$\Longrightarrow$(a): Lemma \ref{yylem3.5}.
\end{proof}

The next lemma, which was stated in \cite[Lemma
4.3]{LWZ} without proof, is easily seen as follows. 

\begin{lemma}
\label{yylem3.7} 
The commutator ideal $[H,H]$ is a Hopf ideal of $H$,
and so $H/[H,H]$ is a commutative Hopf algebra.
\end{lemma}

\begin{proof} Since $\eps$ is a homomorphism from $H$ to a 
commutative ring, its kernel must contain $[H,H]$. It is clear from the
anti-multiplicativity of the antipode that $S([H,H]) \subseteq [H,H]$. 
If $p: H\rightarrow H/[H,H]$ is the quotient map, then $(p\otimes p)
\Delta$ is a homomorphism from $H$ to a commutative ring, so 
$$[H,H] \subseteq \ker (p\otimes p)\Delta= 
\Delta^{-1} \bigl( \ker (p\otimes p) \bigr)= \Delta^{-1} \bigl(
[H,H]\otimes H+ H\otimes [H,H] \bigr).$$
Thus, $\Delta([H,H]) \subseteq [H,H]\otimes H+ H\otimes [H,H]$.
\end{proof}

\begin{proposition}  
\label{yyprop3.8} 
Let $H$ be a Hopf algebra.
\begin{enumerate}
\item
If $e(H)\neq 0$, then 
$H/[H,H]$ is infinite dimensional.
\item 
Suppose $H$ has an affine {\rm(}or noetherian{\rm)} commutative Hopf quotient 
that is infinite dimensional. Then $e(H)\neq 0$.
\item
Suppose $H$ is either affine or noetherian.
Then $H$ satisfies \nat\ if and only if $H/[H,H]$ is
infinite dimensional.
\end{enumerate}
\end{proposition}

\begin{proof} (a) Let $A=H/[H,H]$.
Since $H/\fm^2$ is commutative, there is a surjective
algebra map $A\to H/\fm^2$. This shows that $\ker(\epsilon_A)\neq
\ker(\epsilon_A)^2$, or equivalently $e(A)\neq 0$ by Lemma
\ref{yylem3.1}(a). By Proposition \ref{yyprop3.4} 
(applied to $A$ in place of $H$) and the fact that
$\dim {\mathfrak g}_1=e(A)\neq 0$, the algebra $\gr A\cong U(\mathfrak g)$
is infinite dimensional. Therefore $A$ is infinite dimensional. 

(b) Suppose that $B$ is an affine commutative Hopf quotient of $H$ and that
$B$ is infinite dimensional. Then the 
proof of (a) $\Rightarrow$ (b) in Proposition \ref{yyprop3.6} 
shows that $e(B)=\GKdim B>0$. Since $B$ is a Hopf quotient 
of $H$, $e(H)\geq e(B)>0$.

(c) Note that the condition \nat\ means $e(H)>0$.
The assertion follows from parts (a) and (b).
\end{proof}

Now we consider the case when $H$ has GK-dimension 2 (or
satisfies \Hzr).

\begin{theorem}  
\label{yythm3.9}
Let $H$ be a Hopf algebra of GK-dimension at most 2, 
and assume that $H$ is either
affine or noetherian. Then $H$ satisfies \nat\ if and only if  $H$ has 
a Hopf quotient isomorphic to either $k[t^{\pm1}]$ with $t$ grouplike 
or $k[t]$ with $t$ primitive.
\end{theorem}

\begin{proof} One implication is Proposition \ref{yyprop3.8}(b).
For converse we assume that \nat\ holds. By Proposition \ref{yyprop3.8}(a), 
$H/[H,H]$ is infinite dimensional. By Lemma \ref{yylem3.7}, 
$H/[H,H]$ is a commutative Hopf algebra. Since $H$ is affine or 
noetherian, $H/[H,H]$ has one of these properties. By Molnar's 
theorem \cite{Mol}, $H/[H,H]$ is necessarily affine. Thus, 
$H/[H,H]\cong \cO(G)$ for some infinite algebraic group $G$
of dimension at most $2$. Now
$\cO(G^\circ)$ is a Hopf quotient of $\cO(G)$, where $G^\circ$ is the
connected component of the identity in $G$, and so $\cO(G^\circ)$ is a Hopf
quotient of $H$.
If $\dim G^\circ =1$, then $G^\circ$ is isomorphic to either $k^\times$ or
$k^+$, and the assertion follows. Since $G^\circ$ is infinite, it remains
to consider the case when $\dim G^\circ =2$. In this case, Lemma
\ref{yylem2.2} implies that $G^\circ$ has a one-dimensional closed
connected subgroup, say $G_1$. Then $\cO(G_1)$ is a Hopf quotient of
$\cO(G^\circ)$, isomorphic to $k[t^{\pm1}]$ or $k[t]$.
\end{proof}

\section{Hopf quotients and other preliminaries}
\label{yysec4}

Our analysis of $H$ under \Hzrnat\ is first divided up according to the
GK-dimension of the commutator quotient $H/[H,H]$. This quotient can have
GK-dimension 2 only when $[H,H]=0$, that is, when
$H$ is commutative. That case is covered in Section \ref{yysec2}. In case
$H$ is affine or noetherian, the only other possibility is
$\GKdim H/[H,H] =1$, by Theorem \ref{yythm3.9}.

We now set up some general machinery for Hopf quotients of $H$, which we
will apply later to the quotients arising from Theorem \ref{yythm3.9}.

\begin{subsec}  
\label{yysub4.1}
Let $K\subset H$ be a Hopf ideal, which will be fixed in later sections, 
and let $\Hbar$ be the Hopf quotient algebra $H/K$, with quotient map 
$\pi: H\rightarrow \Hbar$. It is obvious that $H$ is a right and left 
comodule algebra over $\Hbar$ via 
$$\rho=(\id\otimes\pi)\Delta: H\rightarrow H\otimes\Hbar$$ 
and 
$$\lambda=(\pi\otimes\id)\Delta: H\rightarrow \Hbar\otimes H$$
respectively. Using these two maps, we can define two subalgebras
of coinvariants as follows:
$$H_0= H^{\co\rho}:=\{h\in H\mid \rho(h)=h\otimes 1\}$$
and
$${_0H}= {}^{\co\lambda}H:=\{h\in H\mid \lambda(h)= 1\otimes h\}.$$
\end{subsec}

There are various useful relations among $\pi$, $\rho$, $\lambda$ on
the one hand, and the comultiplication and antipode maps on the other, as
follows.

\begin{lemma}  
\label{yylem4.2}
Let $\tau:H\otimes H\rightarrow H\otimes H$ be the flip, and let $S$
denote  the antipodes in both $H$ and $\Hbar$. 
\begin{enumerate}
\item $(\lambda\otimes\id)\Delta = (\pi\otimes\Delta)\Delta=
(\id\otimes\Delta)\lambda$.
\item $(\id\otimes\rho)\Delta = (\Delta\otimes\pi)\Delta=
(\Delta\otimes\id)\rho$.
\item $(\id\otimes\rho)\lambda = (\pi\otimes\rho)\Delta=
(\lambda\otimes\pi)\Delta= (\lambda\otimes\id)\rho$.
\item $\rho S = \tau(S\otimes S)\lambda$.
\item $\lambda S =
\tau(S\otimes S)\rho$.
\end{enumerate}
\end{lemma}

\begin{proof} First, compose the map $\pi\otimes\id\otimes\id$ with 
$(\id\otimes\Delta)\Delta=(\Delta\otimes\id)\Delta$,
to get 
$$(\pi\otimes\Delta)\Delta=(\lambda\otimes\id)\Delta.$$ 
On the other hand,
$$(\pi\otimes\Delta)\Delta= (\id\otimes\Delta)(\pi\otimes\id)\Delta=
(\id\otimes\Delta)\lambda.$$ 
This gives (a), and (b) holds
symmetrically.

Next, composing the map $\pi\otimes\id\otimes\pi$ with 
$(\id\otimes\Delta)\Delta=(\Delta\otimes\id)\Delta$, we have
$$(\pi\otimes\rho)\Delta=(\lambda\otimes\pi)\Delta.$$ 
Since $\pi\otimes\rho=
(\id\otimes\rho)(\pi\otimes\id)$, we have $(\pi\otimes\rho)\Delta=
(\id\otimes\rho)\lambda$, and similarly $(\lambda\otimes\pi)\Delta=
(\lambda\otimes\id)\rho$. This establishes (c).

Finally, if we compose the map $\id\otimes\pi$ with 
$\Delta S= \tau(S\otimes S)\Delta$, we obtain
$$\rho S= \tau(\pi\otimes\id)(S\otimes S)\Delta= \tau(S\otimes
S)(\pi\otimes\id)\Delta= \tau(S\otimes S)\lambda.$$
This yields (d), and (e) holds symmetrically.
\end{proof}

\begin{lemma}  
\label{yylem4.3} 
Retain the notation as above.
\begin{enumerate}
\item 
$\Delta(H_0) \subseteq H\otimes H_0$ and
$\Delta({}_0H) \subseteq {}_0H\otimes H$.
\item
$S(H_0) \subseteq {}_0H$ and $S({}_0H) \subseteq H_0$.
\item
If $H_0= {}_0H$, then $H_0$ is a Hopf subalgebra of $H$.
\item
$H_0= k1+(H_0\cap K)$ and $_0H= k1+({}_0H\cap K)$.
\end{enumerate}
\end{lemma}

\begin{proof} Parts (a)--(c) are generalizations of 
\cite[Proposition 2.1(c,e,h)]{BZ}.
 
(a) If $h\in H_0$, then Lemma \ref{yylem4.2}(b) implies
$$(\id\otimes\rho)\Delta(h)= \Delta(h)\otimes 1.$$ 
Now write $\Delta(h)=\sum_j a_j\otimes b_j$ with the $a_j$ linearly 
independent over $k$. Then $\sum_j a_j\otimes \rho(b_j)= 
\sum_j a_j\otimes b_j\otimes 1$, and consequently $\rho(b_j)= 
b_j\otimes 1$ for all $j$. Thus all $b_j\in H_0$, and so $\Delta(h)
\in H\otimes H_0$, proving the first inclusion. The second is symmetric.

(b) If $h\in H_0$, then in view of Lemma \ref{yylem4.2}(e), $\lambda S(h)= 
\tau(S\otimes S)(h\otimes 1)= 1\otimes S(h)$, so $S(h)\in {}_0H$. 
This establishes the first inclusion. The second is symmetric.

(c) From part (a), we get 
$$\Delta(H_0)\subseteq (H_0\otimes H)\cap (H\otimes
H_0)= H_0\otimes H_0,$$ and from part (b) we get $S(H_0) \subseteq H_0$,
which implies the assertion.

(d) Note that $\rho(H_0)= H_0\otimes 1$ in $H\otimes\Hbar$ implies
$\Delta(H_0)\subseteq H_0\otimes 1+ H\otimes K$ in $H\otimes H$. Hence,
$$H_0= (\eps\otimes\id)\Delta(H_0)\subseteq k1+kK.$$ 
Since $K\subseteq \ker\eps$, it follows that 
$$H_0\cap \ker\eps \subseteq H_0\cap K\subseteq H_0\cap \ker\eps .$$ 
Therefore $H_0= k1+(H_0\cap \ker\eps)= k1+ (H_0\cap K)$, and 
similarly for $_0H$.
\end{proof}

We end this section by recording two known results that will be needed later.

\begin{lemma}  
\label{yylem4.4}
Let $G$ be a totally ordered group
{\rm(}e.g., $G=\ZZ^n${\rm)} and $A$ a $G$-graded domain.
\begin{enumerate}
\item
Every invertible element in $A$ is homogeneous.
\item
Suppose $A$ is a Hopf algebra such that $A\otimes A$ is a domain.
If $x\in A$ such that $\Delta(x)$ is invertible in $A\otimes A$ {\rm(}in
particular, if
$x$ is invertible in $A${\rm)}, then $x\in A_g$ and $\Delta(x)\in
A_g\otimes A_g$ for some
$g\in G$.
\item
Continue the assumptions of part {\rm(b)}.
If, further, $A_g$ is $1$-dimensional, then $x$ is a scalar
multiple of a grouplike element.
\end{enumerate}
\end{lemma}

\begin{proof}
(a) This follows from the hypothesis that $G$ is totally ordered.

(b) The $G$-grading on $A$ induces a $(G\times G)$-grading on 
$A\otimes A$. Since $G\times G$ is totally ordered and
$\Delta(x)$ is invertible, part (a) implies that
$\Delta(x)$ is homogeneous, so $\Delta(x)\in A_g\otimes A_h$
for some $g,h\in G$. By the counit property, $x\in A_g\cap A_h$, and so
$g=h$.

(c) This is clear. 
\end{proof}

\begin{lemma}  
\label{yylem4.5} 
All domains of
GK-dimension $\leq 1$ over $k$ are commutative.
\end{lemma}

\begin{proof} Since $k$ is algebraically closed, the only domain of
GK-dimension zero over $k$ is $k$ itself. The GK-dimension one case is
well known; e.g., it is embedded in the proof of \cite[Theorem 1.1]{ASt}.
It appears explicitly in \cite[Corollary 7.8 (a)$\Rightarrow$(b)]{LWZ}.
\end{proof}

\section{Hopf quotient $\Hbar= k[t^{\pm1}]$ in general}  
\label{yysec5}

From now on through Section \ref{yysec7}, we assume that $H$ has a Hopf
quotient
$\Hbar= H/K=k[t^{\pm1}]$ for some  Hopf ideal $K\subset H$, where $t$ is
a grouplike element  in $\Hbar$. We do not impose the hypotheses \Hzrnat\
until later.

Because $\Hbar= k[t^{\pm1}]$, the right and left
$\Hbar$-comodule algebra structures given by
$\rho$ and $\lambda$ correspond to $\ZZ$-gradings on $H$,
namely 
$$H= \bigoplus_{n\in\ZZ} H_n= \bigoplus_{n\in\ZZ} {}_nH$$
where
$$\begin{aligned}
H_n &= \{h\in H\mid \rho(h)= h\otimes t^n\}  \\
{}_nH &= \{ h\in H \mid \lambda(h)= t^n\otimes h\},
\end{aligned}
$$
(cf.~\cite[Example 4.1.7]{Mon}). 
Write $\pi^r_n$ and $\pi^l_n$ for the respective projections 
from $H$ onto $H_n$ and $_nH$ in the above decompositions. Thus,
$$ \rho(h) = \sum_{n\in\ZZ} \pi^r_n(h)\otimes t^n  \quad{\text{and}}
\quad \lambda(h)= \sum_{n\in\ZZ} t^n\otimes \pi^l_n(h) $$
for $h\in H$.

Note that for
$h\in H$ we have
$$\sum_{m,n\in\ZZ} t^m\otimes \pi^r_n\pi^l_m(h)\otimes t^n=
(\id\otimes\rho)\lambda(h)= (\lambda\otimes\id)\rho(h)= \sum_{m,n\in\ZZ}
t^m\otimes \pi^l_m\pi^r_n(h)\otimes t^n$$
in view of Lemma \ref{yylem4.2}(c), which implies that
\begin{equation}
\label{E5.0.1}\tag{E5.0.1}
\pi^r_n\pi^l_m= \pi^l_m\pi^r_n 
\end{equation}
for all $m,n\in\ZZ$. Consequently, $\pi^l_m(H_n) \subseteq H_n$ and
$\pi^r_n(_mH) \subseteq {}_mH$ for all $m$, $n$, which shows that the
$\lambda$- and $\rho$-gradings on $H$ are {\it compatible\/} in the sense
that
\begin{equation}
\label{E5.0.2}\tag{E5.0.2}
H_n = \bigoplus_{m\in\ZZ} (H_n\cap {}_mH)  \quad
{\text{and}}\quad 
{}_mH =\bigoplus_{n\in\ZZ}\; ({}_mH\cap H_n)  
\end{equation}
for all $m$, $n$.

\begin{lemma}
\label{yylem5.1} 
Retain the notation as above.
\begin{enumerate}
\item
$\Delta\pi^r_n= (\id\otimes\pi^r_n)\Delta$ and 
$\Delta\pi^l_n= (\pi^l_n\otimes\id)\Delta$ for all $n\in\ZZ$.
\item
$\Delta(H_n) \subseteq H\otimes H_n$ and $\Delta({}_nH)\subseteq
{}_nH\otimes H$
 for all $n\in \ZZ$.
\item
$S\pi^r_n= \pi^l_{-n}S$ and $S\pi^l_n= \pi^r_{-n}S$ for all
$n\in \ZZ$.
\item
$S(H_n) \subseteq {}_{-n}H$ and $S({}_nH) \subseteq H_{-n}$ for
all $n\in \ZZ$.  
\end{enumerate}
\end{lemma}

\begin{proof} (a) For $h\in H$, apply Lemma \ref{yylem4.2}(b) to obtain
$$\begin{aligned}  
\sum_{n\in\ZZ} \Delta\pi^r_n(h) \otimes t^n &=
(\Delta\otimes\id)\rho(h)\\
&= (\id\otimes\rho)\Delta(h)= \sum_{(h)} h_1\otimes \rho(h_2)\\
&= \sum_{(h)}\sum_{n\in\ZZ} h_1\otimes \pi^r_n(h_2) \otimes t^n, 
\end{aligned}$$
and consequently $\Delta\pi^r_n(h)= \sum_{(h)} h_1\otimes \pi^r_n(h_2)=
(\id\otimes\pi^r_n)\Delta(h)$ for $n\in \ZZ$. This establishes the first
formula, and the second is symmetric.

(b) If $h\in H_n$, then $\Delta(h)= \Delta\pi^r_n(h)=
(\id\otimes\pi^r_n)\Delta(h) \in H\otimes H_n$. This gives the first
inclusion, and the second is symmetric.

(c) For $h\in H$, apply Lemma \ref{yylem4.2}(e) to obtain
$$\sum_{m\in\ZZ} t^m\otimes \pi^l_mS(h)= \lambda S(h)= \tau(S\otimes
S)\rho(h)= \sum_{n\in\ZZ} t^{-n}\otimes S\pi^r_n(h).$$
This yields $S\pi^r_n(h)= \pi^l_{-n}S(h)$. Similarly, $S\pi^l_n(h)=
\pi^r_{-n}S(h)$.

(d) These inclusions are clear from (c). 
\end{proof}

\begin{lemma}
\label{yylem5.2}
For each $n\in\ZZ$, there is an element $t_n\in H_n\cap {}_nH$ 
such that $\pi(t_n)= t^n$. In particular, $H_n$ and $_nH$ 
are nonzero for all $n$.
\end{lemma}

\begin{proof}
Fix $n$, and choose some $x\in H$ for which $\pi(x)=t^n$. Then 
$$\sum_{i\in\ZZ} \pi\pi^r_i(x)\otimes t^i= (\pi\otimes\id)\rho(x)=
(\pi\otimes\pi)\Delta(x)= \Delta(t^n)= t^n\otimes t^n,$$
whence $\pi\pi^r_n(x)= t^n$. This allows us to replace $x$ by
$\pi^r_n(x)$, that is, we may now assume that $x\in H_n$.

Since the $\lambda$- and $\rho$-gradings are compatible, all 
$\pi^l_j(x)\in H_n$ for all $j\in \ZZ$. As in the previous paragraph, 
$\pi\pi^l_n(x)= t^n$, and we may replace $x$ by $\pi^l_n(x)$. Now 
$x\in H_n\cap {}_nH$, as desired.
\end{proof}

\begin{lemma}
\label{yylem5.3}
The Hopf algebra $H$ is strongly graded with respect to both the
$\rho$-grading and the $\lambda$-grading. 
\end{lemma}

\begin{proof} Let $n\in \ZZ$ and $t_n\in H_n\cap {}_nH$ as in Lemma 
\ref{yylem5.2}. Note that $\eps(t_n)= \eps(t^n) =1$. By Lemma 
\ref{yylem5.1}(b)(c), $\Delta(t_n)= \sum_i a_i\otimes b_i$ for some 
$a_i\in {}_nH$ and $b_i\in H_n$, and then $S(a_i)\in H_{-n}$ and 
$S(b_i)\in {}_{-n}H$ for all $i$. Now
$$1_H = \sum_i S(a_i)b_i \in H_{-n}H_n  \quad {\text{and}}\quad
1_H = \sum_i a_iS(b_i) \in {}_nH{}_{-n}H\,,$$ 
and thus $H_{-n}H_n= H_0$ and ${}_nH{}_{-n}H= {}_0H$. Since $n$ 
was arbitrary, 
it follows that the $\rho$- and $\lambda$-gradings are strong.
\end{proof}

The following results will be helpful in working with gradings on $H$ and
its subalgebras.

\begin{lemma}
\label{yylem5.4}
Let $A= \bigoplus_{n\in\ZZ} A_n$ be a $\ZZ$-graded $k$-algebra 
which is a domain. 
\begin{enumerate}
\item
If $\dim_k A_0<\infty$, and if $A_l\ne 0$ and
$A_{-m}\ne 0$ for some $l,m>0$, then $\dim_k A_w\leq 
1$ for all $w\in\ZZ$, and $\GKdim A = 1$. Further, $A=k[x^{\pm 1}]$ for
some nonzero homogeneous element $x$ of positive degree.
\item
If $A_n\neq 0$ for some $n\neq 0$, then $\GKdim A\geq \GKdim A_0+1$.
\item 
Suppose that $\GKdim A = 1$, and that $A_n\neq 0$ for some $n\neq 0$.
Then $\dim_k A_i \le 1$ for all $i\in\ZZ$. Further, if $\sigma$ is any graded
$k$-algebra automorphism of $A$, there is a nonzero scalar $q\in k$
such that $\sigma(a)= q^ia$ for all $a\in A_i$, $i\in\ZZ$.
\end{enumerate} 
\end{lemma}

\begin{proof} 
(a) Since $k$ is algebraically closed and $A_0$ is a finite dimensional
domain over $k$, we have $A_0=k$.

Since $A$ is a domain, $\dim_k A_w\leq \dim_k A_{w+v}$ for $v,w\in\ZZ$
with
$A_v\neq 0$, because multiplication by any nonzero element of $A_v$
embeds $A_w$ into $A_{w+v}$. Suppose
$A_w\neq 0$ for some
$w\ne 0$. Then $A_{lwm-w} \supseteq A_w^{lm-1} \ne 0$. Further,
$A_{lw(-m)}
\supseteq A_{-m}^{lw} \ne 0$ if $w>0$, while $A_{lw(-m)}
\supseteq A_l^{w(-m)} \ne 0$ if $w<0$. Hence,
$$\dim_k A_w\leq \dim_k A_{lwm}\leq \dim_k A_{lwm+lw(-m)}=\dim_k A_0=1.$$

Now consider an arbitrary finite dimensional subspace $V\subseteq A$, and
note that $V\subseteq \bigoplus_{i=-s}^s A_i$ for some $s>0$.
Consequently, $V^n \subseteq \bigoplus_{i=-ns}^{ns} A_i$ and so $\dim_k
V^n \le 2ns+1$ for all $n>0$. Therefore $\dim_k
V^n$ grows at most linearly with $n$. On the other hand, if $V=kx$ for
some nonzero element $x\in A_l$, then $\dim_k V^n=1$ for all $n>0$.
Therefore
$\GKdim A =1$.

For the last statement, let $x\in A$ be a nonzero homogeneous element of
minimal positive degree and $y\in A$ a nonzero homogeneous element of
largest negative degree. Then $xy$ has degree 0, and we may assume that 
$y=x^{-1}$. If $z$ is a nonzero homogeneous element, then 
$\deg z$ must be a multiple of $\deg x$ by the minimality of
$\deg x$. Therefore $A=k[x^{\pm 1}]$.

(b) Pick a nonzero element $x\in A_n$. Then $x^i\in A_{in}$ for $i\ge 0$, 
from which we see that the nonnegative powers $x^i$ are left linearly 
independent over $A_0$.

Let $B$ be an affine subalgebra of $A_0$, choose a finite dimensional
generating subspace $V$ for $B$ with $1\in V$, and set $W= V+kx$. 
For any $n>0$, the power $W^n$ contains $\sum_{i=0}^n V^{n-i}x^i$, and 
so $\dim_k W^n \ge \sum_{i=0}^n \dim_k V^{n-i}$ because the $x^i$ are 
left linearly independent over $A_0$. Consequently, $\dim_k W^n \ge 
\dim_k Z^n$ for all $n>0$, where $Z=V+kz$ is a finite dimensional 
generating subspace for a polynomial ring $B[z]$. Since $\GKdim B[z]
= 1+\GKdim B$, it follows that $\GKdim A \geq 1+\GKdim B$.

(c) Let $B$ be the graded ring of fractions of $A$ (which exists because
$A$ is commutative). Then
$B$ is a domain over $k$ with
$\GKdim B =1$, and $B_l$ and $B_{-m}$ are nonzero for some $l,m>0$. Part
(b) implies that $\GKdim B_0=0$, whence $B_0=k$. Now part (a) implies that
$B=k[x^{\pm 1}]$ for some nonzero homogeneous element $x$ of positive
degree. In particular, $\dim_k A_i\le \dim_k B_i \le 1$ for all $i\in
\ZZ$. 

Any graded $k$-algebra automorphism $\sigma$ of $A$ extends uniquely to a
graded $k$-algebra automorphism of $B$, which we also denote $\sigma$.
Clearly there is a nonzero scalar $r\in k$ such that $\sigma(x)= rx$. Let
$d=\deg x$, and let $q\in k$ be a $d$-th root of $r$. Then $\sigma(x^l)=
r^lx^l= q^{dl}x^l$ for all $l\in \ZZ$. Since $\deg x^l= dl$, it follows
that $\sigma(b)= q^ib$ for all nonzero $b\in B_i$, $i\in\ZZ$. This
equation holds trivially when $b=0$, yielding the desired description of
$\sigma$.
\end{proof}

\section{$\Hbar=k[t^{\pm 1}]$ and $H_0={_0 H}$}
\label{yysec6}
 
In this section, we classify all affine or noetherian Hopf algebras 
under the hypothesis \Hzrnat\ in the case that $H$ has a Hopf 
quotient $\Hbar=k[t^{\pm 1}]$ with $t$ grouplike and $H_0={_0 H}$. 
The assumptions \Hzrnat, $\Hbar=k[t^{\pm 1}]$, and $H_0={_0 H}$ are 
to hold throughout this section; finiteness conditions are only
imposed in Theorem \ref{yythm6.3}. Recall from Lemma \ref{yylem4.3}(c) 
that $H_0$ must be a Hopf subalgebra of $H$ in these circumstances.

\begin{lemma}
\label{yylem6.1} Either $H_0=k[y]$ with $y$ primitive, or
$H_0\cong k\Gamma$ where $\Gamma$ is a torsionfree abelian group 
of rank $1$.
\end{lemma}

\begin{proof} By Lemma \ref{yylem5.3}, the $\ZZ$-graded algebra 
$H=\bigoplus_n H_n$ is strongly graded. In particular, $H_n\ne 0$ 
for all $n$. Lemma \ref{yylem5.4}(a)(b) then implies that 
$\dim_k H_0= \infty$ and $\GKdim H_0\le 1$, and so by  
Lemma \ref{yylem4.5}, $H_0$ is commutative.
Since $H_0$ is a domain and $k$ is algebraically closed, the only finite
dimensional subalgebra of $H_0$ is $k1$. Thus, $\GKdim H_0=1$, and the
desired conclusion follows from Proposition \ref{yyprop2.1}.
\end{proof}

\begin{lemma}
\label{yylem6.2}
As a $k$-algebra, $H=
H_0[x^{\pm1};\phi]$ where $x\in H_1$ and $\phi$ is a $k$-algebra
automorphism of
$H_0$. Moreover, $x$ can be chosen to be a grouplike element of $H$
satisfying $\pi(x)=t$.
\end{lemma}

\begin{proof} By Lemma \ref{yylem5.3}, $H_1H_{-1}= H_0$, and so there are
elements $x_1,\dots,x_r\in H_1$ and $y_1,\dots,y_r\in H_{-1}$ such that
$\sum_i x_iy_i=1$. In particular, it follows that $H_1= \sum_i x_iH_0$.

In view of Lemma \ref{yylem6.1}, $H_0$ is isomorphic to either a
polynomial ring $k[y]$ or a group ring $k\Gamma$, where $\Gamma$ is free
abelian of rank $1$. In either case,
$H_0$ is a directed union of commutative PIDs, and hence a Bezout domain.
At least one $y_j\ne 0$, and left
multiplication by $y_j$ embeds $H_1$ in $H_0$.
Thus, $H_1$ must be a cyclic right $H_0$-module, say $H_1= xH_0$.

Since $xH_{-1}= xH_0H_{-1}=
H_1H_{-1}= H_0$, we see that $x$ is right invertible, and thus invertible
(because $H$ is a domain). It follows that $H_n= x^nH_0= H_0x^n$ for all
$n\in \ZZ$. In particular, from $xH_0= H_0x$ we see that conjugation by
$x$ restricts to a
$k$-algebra automorphism $\phi$ of
$H_0$. Since $H= \bigoplus_n H_n$ is a free left $H_0$-module with basis
$\{x^n\mid n\in\ZZ\}$, we conclude that $H= H_0[x^{\pm1};\phi]$.

By Lemma \ref{yylem4.3}(d), $\pi(H_0)=k1$, and so $\pi(H_1)= k\pi(x)$. By
Lemma \ref{yylem5.2}, there is some $t_1\in H_1$ such that $\pi(t_1)=t$,
and hence $\pi(x)$ must be a nonzero scalar multiple of $t$. After
replacing
$x$ by a scalar multiple of itself, we may assume that $\pi(x)=t$. In
particular, it follows that
$\epsilon(x)= \epsilon(t) =1$.

Next, observe that $H\otimes H= (H_0\otimes
H_0)[z_1^{\pm1},z_2^{\pm1};\phi_1,\phi_2]$ is an iterated skew-Laurent
polynomial ring over $H_0\otimes H_0$, with respect to the commuting
automorphisms
$\phi_1= \phi\otimes\id$ and $\phi_2= \id\otimes\phi$. Since $H_0\otimes
H_0$ is isomorphic to either a polynomial ring $k[y]\otimes k[y] \cong
k[y,z]$ or a group algebra
$k(\Gamma\times\Gamma)$ over a torsionfree abelian group, it is a domain,
and thus $H\otimes H$ is a domain. Lemma
\ref{yylem4.4}(b) now implies that
$\Delta(x)\in H_n\otimes H_n$ for some $n\in\ZZ$. On the other hand,
$\Delta(x) \in H\otimes H_1$ by Lemma \ref{yylem5.1}(b), so $n=1$ and
$\Delta(x)\in H_1\otimes H_1= xH_0\otimes xH_0$. Thus, $\Delta(x)=
(x\otimes x)w$ for some $w\in H_0\otimes H_0$. Since $\Delta(x)$ and
$x\otimes x$ are invertible in $H\otimes H$, so is $w$. In view of the
skew-Laurent polynomial structure of $H\otimes H$, it follows that $w$ is
invertible in $H_0\otimes H_0$.

If $H_0\cong k[y]$, then $w= \alpha1\otimes1$ for some $\alpha\in \kx$, and
$\Delta(x)= \alpha x\otimes x$. Since $\epsilon(x)=1$, the counit property
implies $x=\alpha x$, forcing
$\alpha=1$. Therefore $x$ is grouplike in this case.

Finally, suppose that $H_0\cong k\Gamma$, so that $H_0\otimes H_0\cong
k(\Gamma\times\Gamma)$. Since the units in $k(\Gamma\times\Gamma)$ are
scalar multiples of elements of $\Gamma\times\Gamma$, we conclude that $w=
\alpha u\otimes v$ for some grouplike elements $u,v\in H_0$. Now
$\Delta(x)= \alpha xu\otimes xv$. The counit property implies $x= \alpha
xu= \alpha xv$, whence $\alpha=1$ in $k$ and $u=v=1$ in $H_0$. Therefore
$\Delta(x)=x\otimes x$, and so
$x$ is grouplike in this case also.
\end{proof}

\begin{theorem}
\label{yythm6.3}
Assume that $H$ is a Hopf algebra satisfying \Hzrnat\ and having a Hopf
quotient $\Hbar=k[t^{\pm 1}]$ with
$t$ grouplike and $H_0={_0 H}$. Then $H$ is affine if and only if $H$ is
noetherian, if and only if $H$ is isomorphic to one of the following Hopf
algebras:
\begin{enumerate}
\item
$A(0,q)$, for some nonzero $q\in k$.
\item
A group algebra $k\Lambda$, where $\Lambda$ is either free abelian of rank
$2$ or the semidirect product $\ZZ\rtimes\ZZ= \langle x,y\mid xyx^{-1}=
y^{-1}\rangle$.
\end{enumerate}
\end{theorem}

\begin{proof} Note first that the Hopf algebras listed in (a) and (b)
are both affine and noetherian. Conversely, assume that $H$ is either
affine or noetherian.

Continue with the notation above, and consider first the case where $H_0=
k[y]$ with $y$ primitive. The automorphism $\phi$ of $H_0$ must send
$y\mapsto qy+b$ for some $q\in\kx$ and $b\in k$, and so $xy=(qy+b)x$.
Applying $\epsilon$ to the latter equation yields $0=b$, so $xy=qyx$. In
view of Lemma \ref{yylem6.2}, we obtain $H\cong A(0,q)$ in this case.

Now assume that $H_0= k\Gamma$ for some torsionfree abelian group $\Gamma$
of rank $1$. Note that since $x$ is grouplike, conjugation by $x$ maps
grouplike elements of $H_0$ to grouplike elements. Hence, $\phi$ restricts
to an automorphism of $\Gamma$, and therefore $H= k\Lambda$ for some
semidirect product $\Lambda= \Gamma\rtimes\ZZ$.

If $H$ is noetherian, then in view of Lemma
\ref{yylem6.2}, $H_0$ must be noetherian. Consequently, $\Gamma$ is infinite
cyclic, and case (b) holds. 

Finally, assume that $H$ is affine, which implies that $\Lambda$ is a
finitely generated group. Note also that $\Lambda$ is solvable. Since
$\GKdim H=2<\infty$, the group $\Lambda$ cannot have exponential growth, and
a theorem of Milnor
\cite[Theorem]{Mil} (or see
\cite[Corollary 11.6]{KL}) shows that $\Lambda$ must be polycyclic. This
forces $\Gamma$ to be polycyclic, whence $\Gamma\cong\ZZ$. Therefore we are
in case (b) again.
\end{proof}

\section{$\Hbar=k[t^{\pm 1}]$ and $H_0\neq {_0 H}$}
\label{yysec7}

In this section, we continue to assume that $H$ satisfies \Hzrnat\ and that
$H$ has a Hopf quotient $\Hbar=k[t^{\pm 1}]$ with $t$ grouplike, but we
now assume that $H_0\neq {_0 H}$. 

\begin{lemma}
\label{yylem7.1} Both $H_0$ and $_0H$ have GK-dimension $1$. 
Consequently, these algebras are commutative.  
\end{lemma}

\begin{proof} Since $\GKdim H>1$, Lemmas \ref{yylem5.2} and 
\ref{yylem5.4}(a) imply that $H_0$ and $_0H$ are infinite dimensional 
over $k$. Since they  are domains, and $k$ is algebraically closed, the 
only finite dimensional subalgebra of either $H_0$ or $_0H$ is $k$. 
Thus, $H_0$ and $_0H$ have GK-dimension at least $1$.

By Lemmas \ref{yylem5.2} and \ref{yylem5.4}(b), $\GKdim H_0 \leq 1$. 
Therefore $\GKdim H_0=1$, and  similarly for $_0H$. The consequence 
follows from Lemma \ref{yylem4.5}.
\end{proof}

In the following results, we will need the notation $H_{i,j}=
H_i\cap{_j H}$, for $i,j\in \ZZ$. In view of \eqref{E5.0.2}, 
$H=\bigoplus_{m,n\in\ZZ} H_{m,n}$ is a $\ZZ^2$-graded algebra. 
Similarly, $H_0= \bigoplus_{n\in \ZZ} H_{0,n}$ and $_0H= 
\bigoplus_{n\in \ZZ} H_{n,0}$ are $\ZZ$-graded algebras. We have 
$H_0 \nsubseteq {}_0H$ or ${}_0H \nsubseteq H_0$, say $H_0 
\nsubseteq {}_0H$. Then $H_{0,n}\ne 0$ for some $n\ne0$. By 
Lemmas \ref{yylem7.1} and \ref{yylem5.4}(c),
$$\dim_k H_{0,i} \le 1$$
for all $i\in\ZZ$. In particular, $\dim_k H_{0,0} \le 1$, and so 
${}_0H \nsubseteq H_0$. Proceeding as with $H_0$, we find that
$$\dim_k H_{i,0} \le 1$$
for $i\in\ZZ$. 

\begin{lemma}
\label{yylem7.2} The intersections $H_{n,n}$ are all
one-dimensional, and
$$\bigoplus_{n\in\ZZ}\, H_{n,n}= k[z^{\pm1}]$$
where $z\in H_{1,1}$ is a grouplike element of $H$ such that
$\pi(z)=t$. Further,
$$H= H_0[z^{\pm 1};\sigma_r]= {}_0H[z^{\pm 1};\sigma_l]$$
where $\sigma_r$ and $\sigma_l$ are graded $k$-algebra automorphisms of
$H_0$ and $_0H$, respectively.
\end{lemma}

\begin{proof}
As noted above, $\dim_k H_{0,i} \le 1$ for all $i\in\ZZ$. Thus, 
$H_{0,0}=k$ and $H_0\otimes H_0= \bigoplus_{m,n\in\ZZ} H_{0,m} \otimes
H_{0,n}$ is a
$\ZZ^2$-graded algebra in which each nonzero homogeneous component is
$1$-dimensional, spanned by a regular element. Since $\ZZ^2$ is an
ordered group, it follows that
$H_0\otimes H_0$ is a domain.

For each $n\in\ZZ$, Lemma \ref{yylem5.2} provides us with a nonzero 
element $t_n\in H_{n,n}$ such that $\pi(t_n)= t^n$. Left 
multiplication by $t_{-n}$ provides an embedding of $H_{n,n}$ 
into $H_{0,0}$. Hence, $H_{n,n}$ must be $1$-dimensional. In 
particular, $H_{1,1}= kz$ where $z=t_1$. Note that 
$\eps(z)= \eps(t) =1$. Now $t_{-1}z$ and $zt_{-1}$ are nonzero 
elements of $H_{0,0} =k$, hence units. Therefore $z$ is a unit in $H$, 
with $z^{-1}\in H_{-1,-1}$. The one-dimensionality of the 
spaces $H_{n,n}$ now forces $H_{n,n}= kz^n$ for all $n$, 
whence $\bigoplus_{n\in\ZZ} H_{n,n}= k[z^{\pm1}]$.

The existence of a unit $z\in H_1$ implies that $H_n= H_0z^n= z^nH_0$ for
all $n$, and thus $H= \bigoplus_{n\in\ZZ} H_0z^n$. Further, $zH_0z^{-1}=
H_0$, and so conjugation by $z$ restricts to a $k$-algebra automorphism
$\sigma_r$ of $H_0$. We conclude that $H= H_0[z^{\pm 1};\sigma_r]$, and 
similarly $H= {}_0H[z^{\pm 1};\sigma_l]$ where $\sigma_l$ is the 
restriction to $_0H$ of conjugation by $z$. Since $z\in {}_1H$, we have
$z({}_nH)z^{-1}= {}_nH$ for all $n$, whence $\sigma_r(H_{0,n})= z(H_0\cap
{}_nH)z^{-1}= H_{0,n}$ for all $n$. This shows that $\sigma_r$ is a
graded automorphism of $H_0$. Similarly, $\sigma_l$ is a graded
automorphism of $_0H$.

Observe that $H\otimes H$ is an iterated skew-Laurent polynomial ring
over $H_0\otimes H_0$, with respect to the commuting automorphisms
$\sigma_r\otimes\id$ and $\id\otimes\sigma_r$. Since $H_0\otimes H_0$ is
a domain, so is $H\otimes H$. Now recall that $H= \bigoplus_{m,n\in\ZZ}
H_{m,n}$ is a
$\ZZ^2$-graded algebra, with
$\dim H_{1,1}=1$ as noted above. Hence, Lemma \ref{yylem4.4}(c) says
that $z$ is a scalar multiple of a grouplike element. Since $\eps(z)=1$,
we conclude that $z$ itself is grouplike.
\end{proof}

The situation of Lemma \ref{yylem7.2} does occur, even in the commutative
case. For instance, let $H= A(n,1)$ with $n\ge1$, and take $K=
\langle y\rangle$ and $t=x+K$. Here $\lambda(y)= t^n\otimes y$ and
$\rho(y)= y\otimes 1$, so $x\in H_0\setminus {}_0H$. The element $z$ of
the  lemma is just $x$.

\begin{lemma}
\label{yylem7.3}
If $H$ is not commutative, 
then either $H_0=H_{0,\geq 0}$ or $H_0=H_{0,\leq 0}$.
\end{lemma}

\begin{proof} Suppose to the contrary that $H_0\neq H_{0,\geq 0}$ 
and $H_0\neq H_{0,\leq 0}$. By Lemmas \ref{yylem7.2} and
\ref{yylem5.4}(a),
$H_0=k[x^{\pm 1}]$ for some nonzero homogeneous element $x$ of positive
degree, say degree $d$. Then $x\in {}_dH$, so $\lambda(x)= t^d\otimes x$,
and consequently the element $\pi(x)\in \Hbar$ satisfies $\Delta\pi(x)=
t^d\otimes \pi(x)$.

Further, we have $H=H_0[z^{\pm 1};\sigma_r]$ as in Lemma
\ref{yylem7.2}. Hence,
$zxz^{-1}= qx^m$ for some $q\in\kx$ and $m\in\{\pm1\}$. Applying
$\Delta\pi$ to the equation $zx=qx^mz$, we obtain
$$(t\otimes t)(t^d\otimes \pi(x))= q(t^{md}\otimes \pi(x)^m)(t\otimes
t),$$ 
from which we see that $m=1$ and $q=1$. But now $H$
is commutative, contradicting our assumptions.
\end{proof}

For the rest of this section, we assume that $H$ is not commutative. By
symmetry, we may further assume that 
$H_0=H_{0,\geq 0}$. This is possible because interchanging $t$ and
$t^{-1}$ in $\Hbar$ results in interchanging $_nH$ and $_{-n}H$ for all
$n$. Recall that $\GKdim H_0=1$ whereas $H_{0,0}=k$, whence $H_{0,j}$ must
be nonzero for infinitely many $j>0$.

We will need the following well known facts about nonzero additive
submonoids $M$ of $\Znonneg$. First, if $n$ is the greatest common
divisor of the positive elements of $M$, then $M= nS= \{ns\mid s\in
S\}$ for a submonoid $S\subseteq \Znonneg$, and the greatest common
divisor of the positive elements of $S$ is $1$. Second, $S$ has a unique
minimal set of generators, consisting of those positive elements of $S$
which cannot be written as sums of two positive elements of $S$. We shall
refer to these elements as the \emph{minimal generators} of $S$. Third,
$S$ has only finitely many minimal generators, and $\Znonneg \setminus S$
is finite.

\begin{lemma}
\label{yylem7.4} 
Set $n=\gcd \{j>0 \mid H_{0,j}\neq 0\}$ and $S= \{i\ge0
\mid H_{0,ni} \neq 0\}$. Let $z\in H_{1,1}$ and $\sigma_r \in \Aut H_0$ 
as in Lemma {\rm\ref{yylem7.2}}.
\begin{enumerate}
\item
There is a scalar $q\neq 0,1$ in $k$ such that
$\sigma_r(h)=q^i h$ for all $h\in H_{0,ni}$, $i\geq 0$.
\item
If $m$ is a minimal generator of $S$, then each $h\in H_{0,nm}$ is a skew
primitive element satisfying
$\Delta(h)= h\otimes1+ z^{nm}\otimes h$.
\item
If $S= \Znonneg$, then $H\cong A(n,q)$.
\item If $S\ne\Znonneg$, then $q$ is a root of unity.
\end{enumerate}
\end{lemma}

\begin{proof} (a) By
Lemma \ref{yylem5.4}(c), there is a nonzero scalar $s\in k$ such that
$\sigma_r(h)= s^jh$ for all $h\in H_{0,j}$, $j\ge 0$. Setting $q= s^n$,
we have $\sigma_r(h)=q^i h$ for all $h\in H_{0,ni}$, $i\geq 0$. 

Now since we have assumed that $H= H_0[z^{\pm1};\sigma_r]$ is not
commutative, $\sigma_r$ is not the identity on $H_0$. But $H_0=
\bigoplus_{i\ge 0} H_{0,ni}$ by choice of $n$, so we must have $q\ne 1$.

(b) Since the spaces $H_{0,ni}$ are at most one-dimensional, we can
choose elements $h_i$ such that $H_{0,ni}= kh_i$ for all $i\ge 0$. In
particular, we choose $h_0=1$. 
Since we are assuming that $H$ is noncommutative, $\sigma_r$ is
nontrivial, and there is some $j>0$ such that $h_j\ne 0$ and $q^j\ne 1$.
Since $zh_j= q^jh_jz$ and $\eps(z)=1$, it follows that $\eps(h_j)=0$. For
all $i>0$, we have $h_i^j \in H_{0,nij}= kh_j^i$, and so $\eps(h_i) =0$.

Now $H_0= \bigoplus_{i\ge0} kh_i$ and $H=
H_0[z^{\pm1};\sigma_r]= \bigoplus_{i\ge0} k[z^{\pm1}]h_i$. Since $z\in
{}_1H$, it follows that $_{ni}H= \bigoplus_{s\ge0} kz^{n(i-s)}h_s$. By
Lemma \ref{yylem5.1}(b), each $\Delta(h_i) \in {}_{ni}H\otimes H_0$, and
so
$$\Delta(h_i)= \sum_{s,t\ge0} \beta^i_{st}z^{n(i-s)}h_s\otimes h_t$$
for some scalars $\beta^i_{st} \in k$. Whenever $h_s=0$ or $h_t=0$, we
choose $\beta^i_{st} =0$. In particular, this means that when
$h_i=0$, we must have $\beta^i_{st} =0$ for all $s$, $t$.
The counit axiom now implies that
$$h_i= \sum_{t\ge0} \beta^i_{0t}h_t= \sum_{s\ge0} \beta^i_{s0}
z^{n(i-s)}h_s \,,$$
from which we see that if $h_i\ne 0$, then $\beta^i_{0i}= \beta^i_{i0}=1$,
while
$\beta^i_{0t} =0$ for $t\ne i$ and $\beta^i_{s0}=0$ for $s\ne i$.
Therefore
$$\Delta(h_i)= h_i\otimes1+ z^{ni}\otimes h_i +u_i 
\qquad\text{where}\qquad u_i= \sum_{s,t\ge1}
\beta^i_{st}z^{n(i-s)}h_s\otimes h_t \,.$$
This formula also holds, trivially, when $h_i=0$.

Assume that $m$ is a minimal generator of $S$. We must show that
$u_m=0$. Observe that
\begin{align*}
h_m\otimes1\otimes1 &+ z^{nm}\otimes
h_m\otimes 1+ z^{nm}\otimes z^{nm}\otimes h_m+ z^{nm}\otimes u_m \\
 &\qquad + \sum_{s,t\ge1} \beta^m_{st} z^{n(m-s)}h_s \otimes \bigl(
h_t\otimes1+ z^{nt}\otimes h_t+ u_t \bigr) \\
 &= (\id\otimes\Delta)\Delta(h_m)= (\Delta\otimes\id)\Delta(h_m) \\
 &= h_m\otimes 1\otimes 1+ z^{nm}\otimes
h_m\otimes 1+ u_m\otimes1+ z^{nm}\otimes z^{nm}\otimes h_m \\
 &\qquad + \sum_{s,t\ge1} \beta^m_{st} (z\otimes z)^{n(m-s)} \bigl(
h_s\otimes1+ z^{ns}\otimes h_s+ u_s \bigr) \otimes h_t \,.
\end{align*}
To analyze the terms of these equations, we view $H\otimes H\otimes H$ as
the direct sum of the spaces $k[z^{\pm1}]h_r \otimes k[z^{\pm1}]h_s
\otimes k[z^{\pm1}]h_t$, which we refer to as the
\emph{$(r,s,t)$-components} of this tensor product. Fix $s,t\ge1$, and
note that the terms
$h_{s'}\otimes u_{t'}$ and $u_{s'}\otimes h_{t'}$ have no
$(s,0,t)$-components. Hence, comparing $(s,0,t)$-components in the above
equation, we find that
$$\beta^m_{st} z^{n(m-s)}h_s \otimes z^{nt}\otimes h_t = \beta^m_{st}
z^{n(m-s)}h_s \otimes z^{n(m-s)} \otimes h_t \,.$$
By our choice of coefficients, $\beta^m_{st}$ cannot be nonzero unless
$h_s$ and
$h_t$ are both nonzero, that is, $s,t\in S$. In this case, we obtain
$z^{nt}= z^{n(m-s)}$ from the equation above. But then $m=s+t$,
contradicting the assumption that
$m$ is a minimal generator of
$S$. Therefore $\beta^m_{st}=0$ for all $s,t\ge 1$, yielding $u_m=0$ as
desired.

(c) If $S= \Znonneg$, then $1$ is the minimal generator of $S$, and $H_0=
k[y]$ where $y= h_1$. By part (b), $y$ is a skew primitive element
satisfying
$\Delta(y)= y\otimes 1+ z^n\otimes y$, and therefore
$H\cong A(n,q)$ in this case.

(d) If $S\ne\Znonneg$, then $1\notin S$ and $S$ has at least two minimal
generators. Choose minimal generators $i<j$. Then $h_i^j$
and $h_j^i$ are nonzero elements of $H_{0,nij}$, and so $h_i^j= \lambda
h_j^i$ for some $\lambda\in\kx$. Since $z^{ni}h_i= q^{ni^2}h_iz^{ni}$ and
$z^{nj}h_j= q^{nj^2}h_jz^{nj}$, we have, because of part (b),
\begin{align*}
\Delta(h_i^j) &= (h_i\otimes1+ z^{ni}\otimes h_i)^j= \sum_{s=0}^j
{j\choose s}_{q^{ni^2}} h_i^{j-s} z^{nis}\otimes h_i^s \ \in\ 
\bigoplus_{s=0}^j H\otimes H_{0,nis}
\\
 \Delta(h_j^i) &= (h_j\otimes1+ z^{nj}\otimes h_j)^i= \sum_{t=0}^i
{i\choose t}_{q^{nj^2}} h_j^{i-t} z^{njt}\otimes h_j^t \ \in\ 
\bigoplus_{t=0}^i H\otimes H_{0,njt}. 
\end{align*}
Note that the expansion of $\Delta(h_j^i)$ contains no term from
$H\otimes H_{0,ni}$, because $0<ni<nj$. Since $\Delta(h_i^j)= \lambda
\Delta(h_j^i)$, it follows that
${j\choose 1}_{q^{ni^2}} =0$, from which we conclude that $q$ must be a
root of unity.
\end{proof}

When  $S\ne \Znonneg$ in Lemma \ref{yylem7.4}, it must have at least two
minimal generators.
This leads us to the Hopf algebras $B(n,p_0,\dots,p_s,q)$, after the
following lemma.

\begin{lemma}
\label{yylem7.5}
Let $\xi\in\kx$. If there is an integer $a\ge2$ such that ${a\choose
r}_\xi =0$ for $0<r<a$, then $\xi$ is a primitive $a$-th root of unity.
\end{lemma}

\begin{proof} The assumption ${a\choose 1}_\xi =0$ implies
that $\xi^{a-1}+ \xi^{a-2} +\cdots+1 =0$, and so $\xi^a=1$. Hence, $\xi$
is a primitive
$b$-th root of unity, for some positive divisor $b$ of $a$. If $a=bc$,
then ${a\choose b}_\xi= {c\choose 1} \ne 0$ (see 
\cite[2.6(iii)]{GLt}, for instance), which is not possible for $b<a$ by
our hypotheses. Therefore $b=a$.
\end{proof}

We are now ready to prove the main result of this section.

\begin{theorem}
\label{yythm7.6} 
Assume that $H$ is a Hopf algebra satisfying \Hzrnat\ and having a Hopf
quotient $\Hbar=k[t^{\pm 1}]$ with
$t$ grouplike and $H_0\neq {_0 H}$. If $H$ is not commutative, 
then $H$ is isomorphic to one of the Hopf algebras $A(n,q)$ or
$B(n,p_0,\dots,p_s,q)$.
\end{theorem}

\begin{proof} We adopt the notation of Lemma \ref{yylem7.4} and its
proof. If $S= \Znonneg$, the lemma shows that $H\cong A(n,q)$, and we are
done. Assume now that $S\ne \Znonneg$, and let $m_1,\dots,m_s$ be the
distinct minimal generators of $S$, arranged in descending order. We must
have $s\ge2$, and Lemma \ref{yylem7.4}(d) shows that $q$ is a root of
unity, say a primitive $\ell$-th root of unity, for some positive integer
$\ell$. Hence, each of the elements
$\xi_i := q^{m_i^2n}$ is a primitive
$p_i$-th root of unity, for some positive integer $p_i$.

As in the proof of Lemma \ref{yylem5.4}(c), the graded ring of
fractions of $H_0$ can be identified with a Laurent polynomial ring
$k[y^{\pm1}]$ where
$y$ is a homogeneous element of degree $n$. Hence, we may choose $h_j=
y^j$ for $j\in S$. Set $y_i= h_{m_i}= y^{m_i}$ for $i=1,\dots,s$, and
recall from Lemma \ref{yylem7.4}(b) that
$$\Delta(y_i)= y_i\otimes1+ z^{m_in}\otimes y_i$$
for such $i$.

Next, let $1\le i<j\le s$. We claim that
\begin{enumerate}
\item $p_i,p_j\ge2$;
\item $m_ip_i= m_jp_j$;
\item $p_i$ and $p_j$ are relatively prime.
\end{enumerate}
There are
relatively prime positive integers $a$ and $b$ such that $m_ia= m_jb$,
and $y_i^a= y_j^b$. Neither of $m_i$ or $m_j$
divides the other (since they are distinct minimal generators of $S$), so
$a,b\ge2$. Since $m_ia= m_jb= \text{lcm}(m_i,m_j)$, none of
$m_i,2m_i,\dots,(a-1)m_i$ can be divisible by $m_j$, and none of
$m_j,2m_j,\dots,(b-1)m_j$ can be divisible by $m_i$. Now 
$$\sum_{r=0}^a {a\choose r}_{\xi_i} y_i^{a-r} z^{m_inr}\otimes y_i^r=
\Delta(y_i^a)= \Delta(y_j^b)= \sum_{t=0}^b {b\choose t}_{\xi_j} y_j^{b-t}
z^{m_jnt}\otimes y_j^t \,.$$
Comparing components in $H\otimes H_{0,nl}$ for $l\ge0$, we find that
${a\choose r}_{\xi_i} =0$ for $0<r<a$ and ${b\choose t}_{\xi_j} =0$ for
$0<t<b$. By Lemma \ref{yylem7.5}, $\xi_i$ is a primitive $a$-th root of
unity and
$\xi_j$ is a primitive $b$-th root of unity. Consequently, $a=p_i$ and
$b=p_j$, from which (a), (b), (c) follow.

Now set $m= m_1p_1= \cdots= m_sp_s= \text{lcm}(m_1,\dots,m_s)$. Since the
$m_i$ were chosen in descending order, the $p_i$ must be in ascending
order: $p_1< p_2< \cdots< p_s$. Further, since the
$p_i$ are pairwise relatively prime, $p_j\mid m_i$ for
all $i\ne j$. Since $m_1,\dots,m_s$ generate $S$, their greatest common
divisor is $1$, and we conclude that $m= p_1p_2\cdots p_s$. Observe that
$q^{m_imn}= q^{m_i^2np_i} =1$ and so $\ell\mid m_imn$ for all $i$, whence
$\ell\mid mn$. For any $i$, this yields $\ell\mid m_inp_i$, whence $\ell=
m'_in'p'_i$ for some positive integers $m'_i\mid m_i$, $n'\mid n$, and
$p'_i\mid p_i$. In particular, $\ell\mid m_i^2np'_i$, and so
$\xi_i^{p'_i}= q^{m_i^2np'_i} =1$. Since $\xi_i$ is a primitive $p_i$-th
root of unity, $p'_i= p_i$. Thus, all $p_i\mid\ell$, and hence
$m\mid\ell$. At this stage, we have $\ell= mn_0$ for some positive
integer $n_0\mid n$. 

Set $p_0=n/n_0$, so that $\ell= (n/p_0)p_1p_2\cdots p_s$. We next show
that $p_0$ is relatively prime to $p_i$ for each $i=1,\dots,s$. Write
$\text{lcm}(p_0,p_i)= p_0c_i$ for some positive integer $c_i$. Then
$n_0p_i$ divides $n_0p_0c_i= nc_i$, whence $\ell$ divides
$nc_i\prod_{j\ne i} p_j^2= m_i^2nc_i$. This implies
$\xi^{c_i}= q^{m_i^2nc_i}=1$ and so $p_i\mid c_i$. It follows that $p_0$
and $p_i$ are relatively prime, as claimed. Therefore $n$, $p_0$, $p_1$,
\dots, $p_s$, $q$ satisfy the hypotheses (a), (b), (c) of Construction
\ref{yycon1.2}. 

As $k$-algebras, $H_0= k[y_1,\dots,y_s] \subset k[y]$ and $H=
H_0[z^{\pm1}; \sigma_r]$ (by Lemma \ref{yylem7.2}), where $\sigma_r$ is
the restriction to $H_0$ of the $k$-algebra automorphism of $k[y]$ such
that $y\mapsto qy$ and $z$ is a grouplike element of $H$. Further,
$y_1,\dots,y_s$ are skew primitive, with $\Delta(y_i)= y_i\otimes1+
z^{m_in}\otimes y_i$. Therefore $H\cong B(n,p_0,\dots,p_s,q)$.
\end{proof}

\section{Hopf quotient $\Hbar= k[t]$ in general}
\label{yysec8}

In this section, \Hzrnat\ is not imposed.

Throughout the section, we assume that $\Hbar=k[t]$ with $t$ primitive.
In this case, the $\Hbar$-comodule algebra structures
$\rho$ and $\lambda$ on $H$ are given by derivations, as follows.

\begin{lemma} 
\label{yylem8.1}
There exist locally nilpotent $k$-linear derivations $\delta_r$ and $
\delta_l$ on $H$ such that
$$\rho(h) = \sum_{n=0}^\infty \dfrac{1}{n!}
\delta_r^n(h)\otimes t^n  
\qquad\quad{\text{and}}\qquad\quad
\lambda(h) = \sum_{n=0}^\infty \dfrac{1}{n!}
t^n\otimes \delta_l^n(h)$$
for $h\in H$. Moreover, $\delta_r$ and $\delta_l$ commute.
\end{lemma}

\begin{proof}
There exist $k$-linear maps $d_0,d_1,\dots$ on $H$ such that
$$\rho(h)= \sum_{n=0}^\infty d_n(h)\otimes t^n$$
for $h\in H$, where $d_n(h)=0$ for $n\gg0$. Since $\eps(t)=0$, the 
counit axiom for $\rho$ implies that $h=d_0(h)$ for $h\in H$, so 
that $d_0= \id_H$. Since $\rho$ is an algebra homomorphism,
$$\sum_{n=0}^\infty d_n(hh')\otimes t^n= \sum_{i,j=0}^\infty 
d_i(h)d_j(h')\otimes t^{i+j}$$
for $h,h'\in H$. Comparing $(-)\otimes t$ terms in this equation, 
we find that 
$$d_1(hh')= d_0(h)d_1(h')+ d_1(h)d_0(h')$$
for $h,h'\in H$. Thus, $\delta_r := d_1$ is a derivation on $H$.

From the coassociativity of $\rho$, we obtain
\begin{align*}
\sum_{i,n=0}^\infty d_id_n(h)\otimes t^i\otimes t^n &= 
(\rho\otimes\id)\rho(h)= (\id\otimes\Delta)\rho(h) \\
 &= \sum_{m=0}^\infty \sum_{i=0}^m {m\choose i} d_m(h)
\otimes t^i\otimes t^{m-i}
\end{align*}
for $h\in H$, whence $d_id_n= {{n+i}\choose i} d_{n+i}$ for all 
$i$, $n$. An easy induction then shows that $d_n= (1/n!)d_1^n$ 
for all $n$. This establishes the asserted formula for $\rho$. 
In particular, it follows that $\delta_r$ is locally nilpotent.

The asserted formula for $\lambda$ is obtained symmetrically.

Since $(\id\otimes\rho)\lambda= (\lambda\otimes\id)\rho$ 
(Lemma \ref{yylem4.2}(c)), we get
$$\sum_{m,n=0}^\infty \dfrac{1}{m!n!} t^m\otimes \delta_r^n\delta_l^m(h)
\otimes t^n= \sum_{m,n=0}^\infty \dfrac{1}{m!n!} t^m\otimes
\delta_l^m\delta_r^n(h) \otimes t^n$$
for $h\in H$. Comparing $t\otimes(-)\otimes t$ terms, we conclude that
$\delta_r\delta_l= \delta_l\delta_r$.
\end{proof}

In particular, it follows from Lemma \ref{yylem8.1} that $H_0= 
\ker\delta_r$ and $_0H= \ker\delta_l$. 

\begin{lemma}
\label{yylem8.2}
\begin{enumerate}
\item
$\Delta\delta_r= (\id\otimes\delta_r)\Delta$ and $\Delta\delta_l=
(\delta_l\otimes\id)\Delta$.
\item
$\Delta(\ker\delta_r^n) \subseteq H\otimes \ker\delta_r^n$ and
$\Delta(\ker\delta_l^n) \subseteq \ker\delta_l^n\otimes H$ for all $n\ge
0$.
\item
$\delta_r S= -S \delta_l$ and $\delta_l S = -S\delta_r$.
\item
$S(\ker\delta_r^n) \subseteq \ker\delta_l^n$ and
$S(\ker\delta_l^n) \subseteq \ker\delta_r^n$ for all $n\ge 0$.
\end{enumerate}
\end{lemma}

\begin{proof} (a) For $h\in H$, apply Lemma \ref{yylem4.2}(b) to obtain
$$\begin{aligned} 
\sum_{n=0}^\infty \frac{1}{n!} \Delta\delta_r^n(h)\otimes t^n &=
(\Delta\otimes\id)\rho(h)= (\id\otimes\rho)\Delta(h)= \sum_{(h)}
h_1\otimes \rho(h_2) \\
 &=\sum_{(h)} \sum_{n=0}^\infty \frac{1}{n!}
h_1\otimes \delta_r^n(h_2)\otimes t^n.  
\end{aligned}$$
It follows that $\Delta\delta_r(h)= \sum_{(h)} h_1\otimes \delta_r(h_2)=
(\id\otimes\delta_r)\Delta(h)$, establishing the first identity. The
second is symmetric.

(b) From (a), we also have $\Delta\delta_r^n=
(\id\otimes\delta_r^n)\Delta$ and $\Delta\delta_l^n=
(\delta_l^n\otimes\id)\Delta$ for all $n\ge0$. Part (b) follows.

(c) For $h\in H$, apply Lemma \ref{yylem4.2}(d) to obtain
$$\sum_{m=0}^\infty \frac{1}{m!} \delta_r^mS(h) \otimes t^m = \rho S(h)=
\tau(S\otimes S)\lambda(h)= \sum_{n=0}^\infty \frac{1}{n!}
S\delta_l^n(h) \otimes (-t)^n.$$
Comparing $(-)\otimes t$ terms, we find that $\delta_r S(h)= -S
\delta_l(h)$.
This establishes the first formula, and the second is symmetric.

(d) From (c), we also have $\delta_r^n S= (-1)^nS
\delta_l^n$ and $\delta_l^n S = (-1)^nS\delta_r^n$ for all $n\ge 0$. Part
(d) follows.
\end{proof}

We include the full hypotheses in the following result because of its
independent interest.

\begin{theorem}
\label{yythm8.3}
Let $H$ be a Hopf algebra over a field $k$ of characteristic zero, with a
Hopf quotient $\Hbar= H/K$ which is a polynomial ring $k[t]$ with 
$t$ primitive. Assume that the rings of coinvariants for the right and
left $\Hbar$-comodule algebra structures $\rho: H\rightarrow H\otimes
\Hbar$ and $\lambda: H\rightarrow \Hbar\otimes H$ coincide. 
\begin{enumerate}
\item
There exists an element $x\in H$ such that $\rho(x)= x\otimes1+ 1\otimes
t$ and $x+K =t$. 
\item
As a $k$-algebra, $H$ is a skew polynomial ring of the
form $H=H_0[x;\partial]$ where $H_0$ is the ring of $\rho$-coinvariants
in $H$ and $\partial$ is a $k$-linear derivation on $H_0$. 
\end{enumerate}
\end{theorem}

\begin{proof} By Lemma \ref{yylem4.3}(c), $H_0$ is a Hopf subalgebra 
of $H$.

Now $H= \bigcup_{n=0}^\infty \ker\delta_r^n$ because $\delta_r$ is
locally nilpotent. If $\ker\delta_r^2= H_0= \ker\delta_r$, then
$\ker\delta_r^{n+1}= \ker\delta_r^n$ for all $n\ge 0$, from which it
would follow that $\ker\delta_r^n = \ker\delta_r$ for all $n\ge0$, and
consequently $H=H_0$. However, there must be some $u\in H$ with
$\pi(u)=t$, and $(\pi\otimes\id)\rho(u)= (\pi\otimes\pi)\Delta(u)=
\Delta(t)= t\otimes1+ 1\otimes t$, whence $\rho(u)\ne u\otimes 1$, that
is, $u\notin H_0$. Thus, $\ker\delta_r^2\ne H_0$.

Now $I := \delta_r(\ker\delta_r^2)$ is a nonzero ideal of $H_0$, because
$\delta_r$ gives an $H_0$-bimodule homomorphism from $\ker\delta_r^2$ to
$H_0$. In view of Lemma \ref{yylem8.2}(a)(b),
$$\Delta(I)= \Delta\delta_r(\ker\delta_r^2)=
(\id\otimes\delta_r)\Delta(\ker\delta_r^2) \subseteq
(\id\otimes\delta_r)(H\otimes\ker\delta_r^2)= H\otimes I.$$
Moreover, $I\subseteq H_0$ and so $\Delta(I)\subseteq H_0\otimes H_0$.
Therefore $\Delta(I)\subseteq H_0\otimes I$.

Choose $v\in (\ker\delta_r^2)\setminus H_0$, and write $\Delta(v)=
\sum_i a_i\otimes b_i$ for some $a_i\in H$ and $b_i\in \ker\delta_r^2$.
Then $\Delta\delta_r(v)= \sum_i a_i
\otimes \delta_r(b_i)$ by Lemma \ref{yylem8.2}(a), and the counit 
axiom yields $0\ne
\delta_r(v)= \sum_i a_i\eps \delta_r(b_i)$. Hence, $\eps\delta_r(b_j) \ne
0$ for some $j$. Since $b_j$ must lie in $(\ker\delta_r^2)\setminus
H_0$, we may replace $v$ by $\eps\delta_r(b_j)^{-1}b_j$, that is, there
is no loss of generality in assuming that $\eps\delta_r(v) =1$.

Continuing with our previous notation, we write $\Delta(v)= \sum_{i=1}^n
a_i\otimes b_i$ for some $a_i\in H$ and $b_i\in \ker\delta_r^2$, but
now assuming that the $b_i$ are $k$-linearly independent. Moreover, we
may assume that there is some $m\le n$ such that $b_i\in H_0$ for $i<m$
while $b_m,\dots,b_n$ are linearly independent modulo $H_0$.
Hence, $\delta_r(b_m),\dots,\delta_r(b_n)$ are linearly independent
elements of $I$. Now 
$$\Delta\delta_r(v)= \sum_{i=m}^n a_i
\otimes \delta_r(b_i)$$
(applying Lemma \ref{yylem8.2}(a) once more).
This sum lies in $H_0\otimes H_0$, because
$\delta_r(v)\in H_0$. Since
$\delta_r(b_m),\dots,\delta_r(b_n)$ are linearly independent elements of
$H_0$, it follows that $a_m,\dots,a_n\in H_0$, whence $S(a_i)\in H_0$
for $i=m,\dots,n$. Consequently, the antipode axiom implies that
$$1_H= \eps\delta_r(v)1_H = \sum_{i=m}^n S(a_i)\delta_r(b_i)
\in I.$$

Therefore, there exists $x\in \ker\delta_r^2$ such that $\delta_r(x)=1$.
Since we may replace $x$ by $x-\eps(x)$, we may also assume that
$\eps(x)=0$. Now $\rho(x)= x\otimes 1+ 1\otimes t$, and so
$$\Delta\pi(x)= (\pi\otimes\pi)\Delta(x)= (\pi\otimes\id)\rho(x)=
\pi(x)\otimes1+ 1\otimes t$$
in $\Hbar\otimes\Hbar$.
Then $\pi(x)= (\eps\otimes\id)\Delta\pi(x)=t$, because $\eps\pi(x)=
\eps(x) =0$. 

Note that because $\rho(x^n)= \rho(x)^n= (x\otimes 1+ 1\otimes t)^n$, we
have $\delta_r^n(x^n)= n!$ for all $n\ge0$. In particular, $x^n\in
\ker\delta_r^{n+1}$. It follows that the $H_0$-bimodule embedding
$(\ker\delta_r^{n+1})/(\ker\delta_r^n) \rightarrow H_0$ induced by
$\delta_r^n$ is an isomorphism, and so
$(\ker\delta_r^{n+1})/(\ker\delta_r^n)$ is a free left or right
$H_0$-module of rank $1$, with the coset of $x^n$ as a basis element.
Therefore $H= \bigoplus_{n=0}^\infty H_0x^n= \bigoplus_{n=0}^\infty
x^nH_0$ as left and right $H_0$-modules.

Finally, since the bimodule isomorphism
$(\ker\delta_r^2)/H_0 \rightarrow H_0$ induced by
$\delta_r$ sends the coset of $x$ to the central element $1$, we
conclude that $xh-hx\in H_0$ for all $h\in H_0$. Therefore the inner
derivation $[x,-]$ restricts to a $k$-linear derivation $\partial$ on
$H_0$, and $H= H_0[x;\partial]$ as a $k$-algebra.  
\end{proof}

\section{$\Hbar= k[t]$ under \Hzrnat}
\label{yysec9}

Assume \Hzrnat, and that $\Hbar=k[t]$ with $t$ primitive.

\begin{lemma}
\label{yylem9.1}
If $A$ is a $k$-algebra supporting a locally nilpotent $k$-linear 
derivation $\delta$ with $\dim_k\ker\delta <\infty$,
then $\GKdim A \le 1$. 
\end{lemma}

\begin{remark*}
No restriction on the base field $k$ is needed for this lemma. 
\end{remark*}

\begin{proof} 
Set $d= \dim_k\ker\delta$, and note that for all $n\ge 2$, the map
$\delta^{n-1}$ induces an embedding $(\ker\delta^n)/(\ker\delta^{n-1})
\rightarrow
\ker\delta$. Hence,
$\dim_k\ker\delta^n \le nd$ for all $n\ge 0$.

Now consider an arbitrary finite dimensional subspace $V\subseteq A$.
Since $\delta$ is locally nilpotent, $V\subseteq \ker\delta^m$ for some
$m>0$. Leibniz' Rule implies that $(\ker\delta^r)(\ker\delta^s)
\subseteq \ker\delta^{r+s}$ for all $r,s>0$, whence $V^n\subseteq
\ker\delta^{nm}$ for all $n>0$. Consequently, $\dim_k V^n \le nmd$ for
all $n>0$, proving that these dimensions grow at most linearly with $n$.
\end{proof}

\begin{lemma}
\label{yylem9.2} If $u\in H\setminus {}_0H$, then $1,u,u^2,\dots$ are
left {\rm(}or right{\rm)} linearly independent over $_0H$. 
\end{lemma}

\begin{proof}
By assumption, there is some $d>0$ such that $\delta_l^d(u)
\ne 0$ while $\delta_l^{d+1}(u) =0$. Then for any positive integer $n$,
we have
$$\lambda(u^n)= \lambda(u)^n= \bigl(\; \sum_{i=0}^d \frac{1}{i!}
t^i\otimes \delta_l^i(u) \bigr)^n= \sum_{j=0}^{nd} t^j\otimes h_j$$
for some $h_j\in H$ with $h_{nd}= \bigl( \frac{1}{d!} \delta_l^d(u)
\bigr)^n\ne 0$. Consequently,
$\delta_l^{nd}(u^n)= (nd)!h_{nd}\ne 0$ while $\delta_l^{nd+1}(u^n)= 0$.

Now if $a_0+ a_1u+ \cdots+ a_nu^n =0$ for some $a_i\in {}_0H$, an
application of $\delta_l^{nd}$ yields $a_n\delta_l^{nd}(u^n) =0$ and
hence $a_n=0$. Continuing inductively yields all $a_i=0$, verifying that
the $u^i$ are left linearly independent over $_0H$. Right linear
independence is proved the same way.
\end{proof}

\begin{lemma}
\label{yylem9.3} $\GKdim H_0 =\GKdim {}_0H= 1$. 
As a consequence, $H_0$ and $_0H$ are commutative. 
\end{lemma}

\begin{proof}
Since $\GKdim H>1$, Lemma \ref{yylem9.1} implies that $H_0$ and $_0H$
are infinite dimensional over $k$. Since they are domains, and $k$
is algebraically closed, the only finite dimensional subalgebra of either
$H_0$ or $_0H$ is $k$. Thus, $H_0$ and $_0H$ have GK-dimension at
least $1$.

Choose an element $u\in H$ such that $\pi(u)=t$. As in the proof of
Theorem \ref{yythm8.3}, we see that $u\notin H_0$, and similarly 
$u\notin {}_0H$. By Lemma \ref{yylem9.2}, the powers $u^i$ are left 
linearly independent over $_0H$. The argument of Lemma \ref{yylem5.4}(b)
therefore shows that $\GKdim {}_0H =1$, and similarly for $H_0$. 
Lemma \ref{yylem4.5} says $H_0$ and $_0 H$ are commutative.
\end{proof}

\begin{proposition}
\label{yyprop9.4} Retain the hypotheses as above. Then $H_0= {}_0H$.
\end{proposition}

\begin{proof} Suppose that $H_0 \nsubseteq {}_0H= \ker\delta_l$. Since
$\delta_l$ commutes with $\delta_r$, it leaves $H_0=\ker\delta_r$
invariant. Thus, $\delta_l$ restricts to a nonzero locally nilpotent
derivation $\delta$ on $H_0$. Set $H_{00}= H_0\cap {}_0H= \ker\delta$, a
subalgebra of $H_0$. 

Choose $u\in H_0\setminus {}_0H$. By Lemma \ref{yylem9.2}, the powers
$u^i$ are left linearly independent over $_0H$ and hence over $H_{00}$.
The argument of Lemma \ref{yylem5.4}(b) then shows that $\GKdim H_0 > \GKdim H_{00}$,
whence $\GKdim H_{00}=0$. Since $H_{00}$ is a domain and $k$ is
algebraically closed, it follows that $H_{00} =k$.

Since $\delta$ is locally nilpotent, some $\delta^d(u)$ will be a nonzero
element of $H_{00}$. Hence, by taking a suitable scalar multiple of
$\delta^{d-1}(u)$, we can obtain an element $y\in H_0$ such that
$\delta(y)=1$. This implies $\lambda(y)= 1\otimes y+ t\otimes 1$, whereas
$\rho(y)= y\otimes 1$ because $y\in H_0$. Set $\ybar= \pi(y)\in \Hbar$
and compute $\Delta(\ybar)$ in the following two ways:
$$\begin{aligned} 
\Delta(\ybar) &= (\pi\otimes\pi)\Delta(y)=
(\pi\otimes\id)\rho(y)= \ybar\otimes1 \\
\Delta(\ybar) &= (\pi\otimes\pi)\Delta(y)= (\id\otimes\pi)\lambda(y)=
1\otimes\ybar+ t\otimes 1. 
\end{aligned}$$
The first equation implies that 
$\ybar= (\eps\otimes\id)\Delta(\ybar)=\eps(\ybar)1_{\Hbar}$. We now have
$$\eps(\ybar)(1\otimes1)=  \Delta(\ybar)= \eps(\ybar)(1\otimes1)+
t\otimes 1$$ 
and consequently $t\otimes1 =0$, which is impossible.

Therefore $H_0\subseteq {}_0H$. The reverse inclusion is obtained
symmetrically. 
\end{proof}

\begin{subsec}  
\label{yysub9.5}
By Proposition \ref{yyprop9.4} and Theorem \ref{yythm8.3}, 
$H=H_0[x;\partial]$  (as a $k$-algebra) for some  $k$-linear 
derivation $\partial$ on $H_0$, where $x$ is an element of $H$ such 
that $\rho(x)= x\otimes1+ 1\otimes t$ and $\pi(x)=t$.

Since $K$ is an ideal of $H$, its contraction $H_0\cap K$ must be a
$\partial$-ideal of $H_0$. Consequently, $(H_0\cap K)H= H(H_0\cap
K)=\bigoplus_{n=0}^\infty (H_0\cap K)x^n$ is an ideal of $H$, contained
in $K$. By Lemma \ref{yylem4.3}(d), $H_0= k1+(H_0\cap K)$. Hence, 
$$H= k[x] \oplus (H_0\cap K)H.$$
Since $\pi$ maps $k[x]$ isomorphically 
onto $k[t]= \Hbar$, the first two equalities below follow. The other 
two are symmetric (or follow from the fact that $H_0={_0 H}$).
\begin{equation}
\label{E9.5.1}\tag{E9.5.1}
K= (H_0\cap K)H= H(H_0\cap K)= ({}_0H\cap K)H= H({}_0H\cap K).
\end{equation}
\end{subsec}

For our main result, we need to know that affineness descends from $H$
to $H_0$ in the present circumstances. This requires the following lemma.

\begin{lemma}  \label{yylem9.6}
Let $A$ be a commutative domain over $k$ with
transcendence degree $1$, and $\delta$ a $k$-linear derivation on $A$.
Assume there exist $a_1,\dots,a_r\in A$ such that $A$ is generated by
$\{\delta^i(a_j)\mid i\ge0,\ j=1,\dots,r\}$. Then $A$ is affine over
$k$. 
\end{lemma}

\begin{remark*}
Here $k$ need not be algebraically closed, but we require $\chr k =0$. 
\end{remark*}

\begin{proof} 
It is known that any $k$-subalgebra of a commutative affine domain with 
transcendence degree $1$ over $k$ is itself affine 
(cf. \cite[Theorem A]{Wa} or \cite[Corollary 1.4]{Al}). Hence, 
it suffices to show that $A$ is
contained in some affine $k$-subalgebra of its quotient field $F$.

Extend $\delta$ to $F$ via the quotient rule. Set
$$A_0= k[a_1,\dots,a_r] \subseteq A_1=
A_0[\delta(a_1),\dots,\delta(a_r)] \subseteq A,$$
and let $F_i= \Fract A_i$ for $i=0,1$, so that $F_0\subseteq F_1
\subseteq F$. Note that $\delta(A_0)\subseteq A_1$, whence
$\delta(F_0)\subseteq F_1$. Because of characteristic zero, any element
of $F$ which is algebraic over $k$ must lie in $\ker\delta$. Hence, if
all of the $a_j$ were algebraic over $k$, we would have $\delta=0$ and
$A=A_0$ finite dimensional over $k$, contradicting the assumption of
transcendence degree $1$. Thus, at least one $a_j$ is transcendental
over $k$. It follows that $F_0$ has transcendence degree $1$ over $k$,
and therefore $F$ must be algebraic over $F_0$.

Consider a nonzero element $u\in F_1$, and let $f(z)= z^n+
\alpha_{n-1}z^{n-1} +\cdots+ \alpha_0 \in F_0[z]$ be its minimal
polynomial over $F_0$. Applying $\delta$ to the equation $f(u)=0$, we
obtain
$$\bigl[ \delta(\alpha_{n-1})u^{n-1} +\dots+ \delta(\alpha_1)u+
\delta(\alpha_0)
\bigr] +
\bigl[ nu^{n-1}+ (n-1)\alpha_{n-1}u^{n-2}+ \cdots+ \alpha_1 \bigr]
\delta(u) =0,$$
from which we see that $\delta(u) \in F_1$. Thus, $\delta(F_1) \subseteq
F_1$.

In particular, each $\delta^2(a_j)\in F_1$, and so there is a nonzero
element $v\in A_1$ such that $\delta^2(a_j) \in A_1v^{-1}$ for all $j$.
It follows that the algebra $A_1[v^{-1}]$ is stable under $\delta$.
Therefore $A$ is contained in the affine $k$-algebra $A_1[v^{-1}]$,
completing the proof.
\end{proof}

\begin{proposition}  
\label{yyprop9.7}
Retain the hypotheses as above, and assume that $H$ is affine or
noetherian. Then either $H_0= k[y^{\pm1}]$ with $y$ grouplike or $H_0= k[y]$
with $y$ primitive.
\end{proposition}

\begin{proof} If $H$ is commutative, then it is affine, and the result
follows from Proposition \ref{yyprop2.3}, with an appropriate choice 
of the Hopf ideal $K$. In case (a) of the proposition, $H= k[x,y]$ 
with $x$ and $y$ primitive. Take $K=\langle y\rangle$, and observe 
that $H_0= k[y]$ iin this case. Case (b) does not occur, since 
$k\Gamma$ has no Hopf quotient isomorphic to $k[t]$. In case (c), 
we again take $K= \langle y\rangle$ and observe that $H_0= k[y]$. 

We now assume that $H$ is
not commutative. Since $H_0$ is a commutative domain of GK-dimension $1$,
it will suffice to show that $H_0$ is affine, by Proposition
\ref{yyprop2.1}. Write $H= H_0[x;\partial]$ as in \S\ref{yysub9.5}, and
note that because of our noncommutativity assumption, $\partial\ne0$. If $H$
is noetherian, then so is $H_0$, whence $H_0$ is affine by \cite{Mol}. Now
suppose that
$H$ is affine.

By hypothesis, $H$ can be generated by finitely many elements
$h_1,\dots,h_m$. Each $h_j$ is a finite sum of terms $a_{jl}x^l$ for
$a_{jl}\in H_0$ and $l\ge0$. Hence, there is a finite sequence
$a_1,\dots,a_r$ of elements of $H_0$ such that $H$ can be generated by
$x,a_1,\dots,a_r$. 

Consider the $k$-subalgebra
$$A= k[\partial^i(a_j)\mid i\ge0,\ j=1,\dots,r] \subseteq H_0 \,,$$
and note that $\partial(A)\subseteq A$. Now
$\bigoplus_{n=0}^\infty Ax^n$ is a subalgebra of $H$. Since
it contains
$x$ and the
$a_j$, it must equal $H$. From $\bigoplus_{n=0}^\infty Ax^n=
\bigoplus_{n=0}^\infty H_0x^n$, we conclude that $A=H_0$. Therefore
$H_0$ is affine by Lemma \ref{yylem9.6}.
\end{proof}

\section{$\Hbar= k[t]$ and $H_0=k[y]$}
\label{yysec10}

Assume \Hzrnat\ throughout this section, and additionally that $\Hbar=
k[t]$ and $H_0=k[y]$ with $t$ and
$y$ primitive. We also assume that $H$ is not commutative.

\begin{subsec}  
\label{yysub10.1}
Lemma \ref{yylem4.3}(d) implies that $H_0\cap K= H_0\cap \ker\eps=
yH_0$, and so by \eqref{E9.5.1} we get $K= yH= Hy$. By \S\ref{yysub9.5}, we
also have
$H= H_0[x;\partial]$ where $\pi(x)=t$  and $\rho(x)= x\otimes1 +
1\otimes t$; in particular, $\eps(x)=0$ and $\delta_r(x)=1$. For any
$\alpha\in\kx$, we have $\pi(\alpha x)= \alpha t$ and $\rho(\alpha x)=
\alpha x\otimes1+ 1\otimes\alpha t$, while $H= H_0[\alpha x;
\alpha\partial]$ and $\Hbar= k[\alpha t]$ with $\alpha t$ primitive.
Thus, we may replace $x$ by any nonzero  scalar multiple of itself, 
as long as we
replace $t$ and $\partial$ by corresponding multiples of themselves.
Similarly, we may add any element of $H_0\cap K$ to $x$.

Now $xy-yx=\partial(y)= f$ for some nonzero $f\in H_0$. Note that $f(0)=
\eps(f)= \eps([x,y]) =0$, so that $f$ is divisible by $y$. 
Moreover, $\partial= f\frac{d}{dy}$, and so 
$fH_0$ is a $\partial$-ideal of $H_0$. Hence, $fH=Hf$ is an ideal of $H$,
and $H/fH$ is commutative. Since $f=[x,y]$, it follows that $fH= [H,H]$.
By Lemma \ref{yylem3.7}, $fH$ is thus a Hopf ideal of $H$. In particular,
$$\Delta(f)\in fH\otimes H+ H\otimes fH= \bigoplus_{i,j=0}^\infty \bigl[
fH_0x^i\otimes H_0x^j+ H_0x^i\otimes fH_0x^j \bigr].$$
On the other hand, $\Delta(f) \in \Delta(H_0) \subseteq H_0\otimes H_0$,
from which we conclude that $\Delta(f)\in fH_0\otimes H_0+ H_0\otimes
fH_0$.

Write $f= \alpha_1y+ \cdots+ \alpha_ry^r$ for some $\alpha_i\in
k$ with $\alpha_r\ne 0$, and note that the cosets of the tensors
$y^i\otimes y^j$ for $0\le i,j\le r-1$ form a basis for the quotient
$(H_0\otimes H_0)/(fH_0\otimes H_0+ H_0\otimes
fH_0)$. Computing modulo $fH_0\otimes H_0+ H_0\otimes
fH_0$, we obtain
$$\begin{aligned} 
0 &\equiv \Delta(f)= \sum_{l=1}^r \alpha_l(y\otimes1+ 1\otimes y)^l\\
&= \sum_{\substack{i,j\ge0\\ 0<i+j\le r}} {{i+j}\choose
i} \alpha_{i+j} y^i\otimes y^j \\
 &\equiv \sum_{\substack{r>i,j\ge0\\ 0<i+j\le r}}
{{i+j}\choose i} \alpha_{i+j} y^i\otimes y^j - \sum_{m=1}^{r-1}
\alpha_my^m\otimes 1 -
\sum_{n=1}^{r-1} 1\otimes \alpha_ny^n.  
\end{aligned} $$
If $r\ge2$, the last line above has only one term involving $y\otimes
y^{r-1}$, and this term has the coefficient ${r\choose 1}\alpha_r\ne 0$,
which is impossible. Hence, $r=1$, and $f= \alpha_1y$ with $\alpha_1\ne
0$.

At this point, we replace $x$ with $\alpha_1^{-1}x$, so that $xy-yx =y$,
that is,
\begin{equation}
\partial= y \frac{d}{dy}\,.  \label{E10.1.1} \tag{E10.1.1}
\end{equation}
\end{subsec}

\begin{subsec}  \label{yysub10.2}
Write $\Delta(x)= x\otimes1+ 1\otimes x +u$ for
some $u\in H\otimes H$, and observe that
$$\begin{aligned} 
(\id\otimes\rho)\Delta(x) &= x\otimes1\otimes1+ 1\otimes
x\otimes 1+ 1\otimes1\otimes t +(\id\otimes\rho)(u)  \\
(\Delta\otimes\id)\rho(x) &= x\otimes1\otimes1+ 1\otimes
x\otimes 1+ u\otimes 1+ 1\otimes1\otimes t.  
\end{aligned}
$$
In view of Lemma \ref{yylem4.2}(b), we obtain 
$(\id\otimes\rho)(u)= u\otimes1$, from which
it follows that $u\in H\otimes H_0$. We also have
$$\begin{aligned} 
(\id\otimes\Delta)\Delta(x) &= x\otimes1\otimes1+ 1\otimes
x\otimes 1+ 1\otimes1\otimes x+ 1\otimes u +(\id\otimes\Delta)(u)  \\
(\Delta\otimes\id)\Delta(x) &= x\otimes1\otimes1+ 1\otimes
x\otimes 1+ u\otimes 1+ 1\otimes1\otimes x+ (\Delta\otimes\id)(u), 
\end{aligned}
$$
and so coassociativity of $\Delta$ implies that
\begin{equation}
\label{E10.2.1}\tag{E10.2.1}
1\otimes u +(\id\otimes\Delta)(u)= u\otimes 1+ (\Delta\otimes\id)(u). 
\end{equation}
The counit axioms imply that
$$x = x+m(\id\otimes\eps)(u)  \quad{\text{and}}\quad
x = x+m(\eps\otimes\id)(u),
$$ 
and consequently
\begin{equation}
\label{E10.2.2}\tag{E10.2.2}
m(\id\otimes\eps)(u)= m(\eps\otimes\id)(u)= 0. 
\end{equation}

Next, observe that
$$\begin{aligned} y\otimes1+ 1\otimes y &= \Delta(y)= \Delta([x,y])=
[\Delta(x),\Delta(y)] \\
 &= y\otimes1+ 1\otimes y+ [u,\Delta(y)],  
\end{aligned}$$
whence $[u,\Delta(y)] =0$. Since $u$ lies in $H\otimes H_0$, it commutes
with $1\otimes y$, and thus it must also commute with $y\otimes 1$. If we
write $u= \sum_j u_j\otimes y^j$ with the $u_j\in H$, then we find that
each $u_j$ must commute with $y$. As is easily checked, the centralizer
of $y$ in $H$ is just $H_0$. Therefore $u\in H_0\otimes H_0$. This allows
us to write
$$u=\sum_{i,j\ge0} \beta_{ij}y^i\otimes y^j$$
for some $\beta_{ij}\in k$. Equation \eqref{E10.2.2} becomes
$$\sum_{i\ge0} \beta_{i0}y^i= \sum_{j\ge0} \beta_{0j}y^j =0,$$
whence $\beta_{i0}= \beta_{0j} =0$ for all $i$, $j$. Thus,
$u=\sum_{i,j\ge1} \beta_{ij}y^i\otimes y^j$. 

At this point, \eqref{E10.2.1} becomes
$$\begin{aligned}
\sum_{i,j\ge1} \beta_{ij}\otimes y^i\otimes y^j+&
\sum_{i,j\ge1} \beta_{ij}y^i\otimes (y\otimes1+ 1\otimes y)^j \\
 &= \sum_{i,j\ge1} \beta_{ij}y^i\otimes y^j\otimes1+ \sum_{i,j\ge 1}
\beta_{ij}(y\otimes1+ 1\otimes y)^i\otimes y^j,  \end{aligned}$$
and consequently
\begin{equation}
\label{E10.2.3}\tag{E10.2.3}
\sum_{i,j\ge1} \sum_{l=0}^{j-1} {j\choose l} \beta_{ij}y^i\otimes
y^l\otimes y^{j-l}= \sum_{i,j\ge1} \sum_{m=0}^{i-1} {i\choose m}
\beta_{ij} y^{i-m}\otimes y^m\otimes y^j.  
\end{equation}
Comparing terms in \eqref{E10.2.3} involving 
$y^{i-1}\otimes y\otimes y^{j-1}$, we
see that $j\beta_{i-1,j}= i\beta_{i,j-1}$ for all $i,j\ge 2$. It follows
that
$$\begin{aligned} \beta_{ij} &= \frac{j+1}{i}\beta_{i-1,j+1}=
\frac{(j+1)(j+2)}{i(i-1)}
\beta_{i-2,j+2}= \cdots  \\
 &= \frac{(j+1)\cdots(j+i-1)}{i!}\beta_{1,j+i-1}
 = \frac{1}{i+j}{{i+j}\choose i} \beta_{1,j+i-1}  
\end{aligned}$$
for all $i,j\ge2$, and similarly $\beta_{ij}= \frac{1}{i+j}{{i+j}\choose
j} \beta_{i+j-1,1}$. Set $\gamma_s= (1/s)\beta_{1,s-1}$ for all
$s\ge 2$, so that $\beta_{ij}= {{i+j}\choose i} \gamma_{i+j}$ for all
$i,j\ge1$. We now have
\begin{equation}
\begin{aligned}
u &= \sum_{i,j\ge1} {{i+j}\choose i} \gamma_{i+j} y^i\otimes
y^j  \\
 &= \sum_{s\ge1} \gamma_s \bigl[ (y\otimes1+ 1\otimes y)^s -
y^s\otimes1- 1\otimes y^s \bigr]  \notag \\
 &= g(y\otimes1+ 1\otimes y)- g\otimes1- 1\otimes
g,  \end{aligned}  \label{E10.2.4}\tag{E10.2.4}
\end{equation}
where $g= \sum_{s\ge1} \gamma_sy^s\in H_0$. Note that since $g$ has zero
constant term, $g\in K$. 

The element $x'= x-g\in H$ satisfies $\pi(x')=t$ and $\rho(x')=
x'\otimes1+ 1\otimes t$, as well as $x'y-yx'=y$. In view of \eqref{E10.2.4}
$$\Delta(x')= x\otimes1+ 1\otimes x+ u- g(y\otimes1+ 1\otimes y)=
x'\otimes1+ 1\otimes x'.$$
Thus, we may replace $x$ by $x'$, so that
$$\Delta(x)= x\otimes1+ 1\otimes x. $$
Therefore $x$ is primitive.
\end{subsec}

\begin{theorem}
\label{yythm10.3}
Assume \Hzrnat, and suppose that $\Hbar= k[t]$ and $H_0=k[y]$ with
$t$ and $y$ primitive. If $H$ is not commutative, then $H\cong U(\gfrak)$
where $\gfrak$ is a
$2$-di\-men\-sion\-al nonabelian Lie algebra over $k$.
\end{theorem}

\begin{proof} The work of
Subsections \ref{yysub10.1} and
\ref{yysub10.2} shows that there is a primitive element $x\in H$ such
that $H= k[y]\bigl[ x; y\tfrac{d}{dy}
\bigr]$. Thus, $H\cong U(\gfrak)$ where $\gfrak= kx+ky$ is a
$2$-di\-men\-sion\-al nonabelian Lie algebra over $k$.
\end{proof}

\section{$\Hbar=k[t]$ and $H_0=k[y^{\pm1}]$}
\label{yysec11}

In this section, we assume \Hzrnat, and also that $\Hbar=k[t]$ and
$H_0=k[y^{\pm1}]$ with
$t$ primitive and $y$ grouplike. As in the previous section, we also assume
that $H$ is not commutative.

\begin{subsec} \label{yysub11.1}
Lemma \ref{yylem4.3}(d) implies that $H_0\cap K= H_0\cap \ker\eps=
(y-1)H_0$, and so this time \eqref{E9.5.1} tells us that 
$K= (y-1)H= H(y-1)$. In particular, $\pi(y)=1$. From \S\ref{yysub9.5}, we
have $H= H_0[x;\partial]$ for $x\in H$ with 
$\pi(x)=t$ and $\rho(x)= x\otimes1 +1\otimes t$. As remarked in the 
beginning of the last section, we may replace $x$ by a nonzero
scalar multiple of itself, and we may add an element of $H_0\cap K$ to
$x$. Since $\pi(y)=1$, we may also replace $x$ by a power of $y$ times 
$x$. Because $\eps(t)=0$, we always have $\eps(x)=0$.

Now $xy-yx=f$ for some nonzero $f\in H_0$, and $f(1)= \eps(f) =0$. 
Similarly to the last section, we see that $fH=[H,H]$ is a Hopf ideal 
of $H$, and then that $\Delta(f)\in fH_0\otimes H_0+ H_0\otimes fH_0$.

Write $f=y^r(\alpha_0+\alpha_1y+ \cdots+ \alpha_sy^s)$ for some $r\in
\ZZ$ and $s\in\Zpos$, and some $\alpha_i\in k$ with $\alpha_0,\alpha_s \ne
0$. (We must have $s>0$ because $f(1)=0$.) The cosets of the tensors
$y^i\otimes y^j$ for $0\le i,j\le s-1$ form a basis for $(H_0\otimes
H_0)/ (fH_0\otimes H_0+ H_0\otimes fH_0)$. Computing modulo $fH_0\otimes
H_0+ H_0\otimes fH_0$, we obtain
$$0\equiv (y^{-r}\otimes y^{-r})\Delta(f)= \sum_{i=0}^s
\alpha_iy^i\otimes y^i \equiv \sum_{i=0}^{s-1} \alpha_iy^i\otimes y^i +
\alpha_s^{-1} \sum_{l,m=0}^{s-1} \alpha_ly^l \otimes \alpha_my^m .$$

Looking at $1\otimes1$ terms, we see that
$\alpha_0+\alpha_s^{-1}\alpha_0^2=0$, whence $\alpha_0=-\alpha_s$.
Looking at $y^l\otimes1$ terms for $0<l<s$, we see that
$\alpha_s^{-1}\alpha_l\alpha_0=0$, whence $\alpha_l=0$. Hence,
$f=\alpha_sy^r(y^s-1)$. 

At this point, we replace $x$ by $\alpha_s^{-1}y^{1-r}x$, which replaces
$f$ by $y^n-y$, where $n=s+1\ge2$. Thus,
\begin{equation}
\partial= (y^n-y)\frac{d}{dy}\,, \qquad n\in \ZZ_{\ge2} \,. 
\tag{E11.1.1} \label{E11.1.1}
\end{equation}
\end{subsec} 

\begin{subsec} \label{yysub11.2}
Write $\Delta(x)= x\otimes y^{n-1}+ 1\otimes x +u$ for
some $u\in H\otimes H$, and observe that
$$\begin{aligned} 
(\id\otimes\rho)\Delta(x) &= x\otimes y^{n-1}\otimes1+ 1\otimes
x\otimes 1+ 1\otimes1\otimes t +(\id\otimes\rho)(u)  \\
(\Delta\otimes\id)\rho(x) &= x\otimes y^{n-1}\otimes1+ 1\otimes
x\otimes 1+ u\otimes 1+ 1\otimes1\otimes t.  
\end{aligned}$$
In view of Lemma \ref{yylem4.2}(b), we obtain $(\id\otimes\rho)(u)=
u\otimes1$,  from which it follows that $u\in H\otimes H_0$. We also have
\begin{multline}
(\id\otimes\Delta)\Delta(x) =  \\
 x\otimes y^{n-1}\otimes y^{n-1}+
1\otimes x\otimes y^{n-1}+ 1\otimes1\otimes x+ 1\otimes u
+(\id\otimes\Delta)(u)   \notag
\end{multline}  
 \begin{multline} 
(\Delta\otimes\id)\Delta(x) =  \\
 x\otimes y^{n-1}\otimes y^{n-1}+ 1\otimes
x\otimes y^{n-1}+ u\otimes y^{n-1}+ 1\otimes1\otimes x+
(\Delta\otimes\id)(u),   \notag
\end{multline}
and so coassociativity of $\Delta$ implies that
\begin{equation}
\label{E11.2.1}\tag{E11.2.1}
1\otimes u +(\id\otimes\Delta)(u)= u\otimes
y^{n-1}+ (\Delta\otimes\id)(u).  
\end{equation}

The counit axioms imply that
$$ x = x+m(\id\otimes\eps)(u)  \quad{\text{and}}\quad
x = x+m(\eps\otimes\id)(u),
$$ 
and consequently
\begin{equation}
\label{E11.2.2}
\tag{E11.2.2}
m(\id\otimes\eps)(u)= m(\eps\otimes\id)(u)= 0.  
\end{equation}

Next, observe that
$$\begin{aligned} 
y^n\otimes y^n- y\otimes y &= \Delta(y^n-y)= \Delta([x,y])=
[\Delta(x),\Delta(y)]  \\
 &= [x\otimes y^{n-1}+ 1\otimes x +u,\, y\otimes y]  \\
 &= (y^n-y)\otimes y^n+ y\otimes (y^n-y)+ [u,\, y\otimes y],
\end{aligned}
$$ 
whence $[u,\, y\otimes y] =0$. Since $u$ lies in $H\otimes H_0$, it
commutes with $1\otimes y$, and thus it must also commute with $y\otimes
1$. As in \S\ref{yysub10.2}, this forces $u\in H_0\otimes H_0$,
and so
$$u= \sum_{i,j\in\ZZ} \beta_{ij}y^i\otimes y^j$$
for some $\beta_{ij}\in k$. Equation \eqref{E11.2.2} becomes
$$\sum_{i,j} \beta_{ij}y^i= \sum_{i,j} \beta_{ij}y^j= 0,$$
whence 
\begin{equation}
\label{E11.2.3}\tag{E11.2.3}
\sum_j \beta_{ij} =0 \quad (\text{all\ } i) 
\quad{\text{and}}\quad
\sum_i \beta_{ij} =0 \quad (\text{all\ } j).
\end{equation}

At this point, \eqref{E11.2.1} becomes
\begin{equation}
\begin{aligned}
\sum_{i,j} \beta_{ij}\otimes y^i\otimes y^j&+ \sum_{i,j}
\beta_{ij}y^i\otimes y^j\otimes y^j  \\
& =\sum_{i,j} \beta_{ij}y^i\otimes y^j\otimes y^{n-1}+ \sum_{i,j}
\beta_{ij}y^i\otimes y^i\otimes y^j. \notag
\end{aligned}  \label{E11.2.4}\tag{E11.2.4}
\end{equation}

Comparing terms in \eqref{E11.2.4} involving $1\otimes1\otimes1$, we see
that
$\beta_{00} =0$, while comparing terms involving $1\otimes y^{n-1}\otimes
y^{n-1}$ yields $\beta_{n-1,n-1} =0$. Turning to other terms $1\otimes
y^j\otimes y^j$, we obtain $\beta_{jj}+\beta_{0j} =0$ for all $j\ne
0,n-1$, while inspection of terms involving  $1\otimes y^i\otimes
y^{n-1}$ yields $\beta_{i,n-1}= \beta_{0i}$ for all $i\ne 0,n-1$.
Finally, looking at other terms $1\otimes y^i\otimes y^j$, we conclude
that
$\beta_{ij}=0$ whenever $i\ne 0,j$ and $j\ne n-1$. Set $\gamma_s=
\beta_{s,n-1}$ for all $s$. The only coefficients $\beta_{ij}$ which
could be nonzero are the following:
$$ \begin{aligned}
\beta_{0j} &= \begin{cases} \gamma_j & (j\ne 0,n-1)\\
\gamma_0 & (j=n-1) \end{cases} \\
\beta_{i,n-1} &= \gamma_i \quad (i\ne n-1)\\
\beta_{jj} & = -\gamma_j \quad(j\ne 0,n-1).
\end{aligned}
$$
Moreover, \eqref{E11.2.3} implies that $\sum_s \gamma_s =0$. Now
\begin{equation}
\begin{aligned}
u &= \sum_{j\ne 0,n-1} \gamma_j\otimes y^j+ \gamma_0\otimes
y^{n-1}+ \sum_{i\ne0,n-1} \gamma_iy^i\otimes y^{n-1}  \\
  &\qquad\qquad\quad - \sum_{j\ne 0,n-1}
\gamma_jy^j\otimes y^j \\
 &= \gamma_0\otimes y^{n-1}+ 1\otimes g+ g\otimes y^{n-1}- g(y\otimes
y),
\end{aligned}  \label{E11.2.5}\tag{E11.2.5} 
\end{equation}
where $g= \sum_{j\ne0,n-1} \gamma_jy^j \in H_0$. Note that
$g(1)+\gamma_0= \sum_j \gamma_j
=0$, whence $g+\gamma_0\in H_0\cap K$.

The element $x'= x+g+\gamma_0 \in H$ satisfies $\pi(x')= t$ and
$\rho(x')= x'\otimes 1+1\otimes t$, as well as $x'y-yx'= y^n-y$. In view
of \eqref{E11.2.5},
$$\Delta(x')= x\otimes y^{n-1}+ 1\otimes x +u +g(y\otimes y)+
\gamma_0\otimes 1= x'\otimes y^{n-1}+ 1\otimes x'.$$
Thus, we may replace $x$ by $x'$, so that
$$\Delta(x)= x\otimes y^{n-1}+ 1\otimes x. $$
Therefore $x$ is skew primitive.
\end{subsec}

The work in Subsections \ref{yysub11.1} and \ref{yysub11.2} yields
the following theorem.

\begin{theorem} 
\label{yythm11.3}
Assume \Hzrnat, and suppose that $\Hbar= k[t]$ and $H_0= k[y^{\pm1}]$ with
$t$ primitive and $y$ grouplike. If $H$ is not commutative, then $H\cong
C(n)$ for some integer $n\ge2$.  \qed
\end{theorem}

\section{Proof of Theorem \ref{yythm0.1}} 
\label{yysec12}

Observe first that the algebras listed in Theorem \ref{yythm0.1}(I)--(V)
are affine and noetherian. They are pairwise non-isomorphic, up to 
$A(0,q)\cong A(0,q^{-1})$,  by Proposition \ref{yyprop1.6}.

Now assume that $H$ is a Hopf algebra satisfying \Hzrnat, and that $H$ is
either affine or noetherian. If $H$ is commutative, then it is necessarily
affine \cite{Mol}, and Proposition \ref{yyprop2.3} implies that $H$ is
isomorphic to a Hopf algebra of type (Ia), (IIa), or (III) (with $q=1$). 
It remains to consider the case that $H$ is not commutative.

By Theorem \ref{yythm3.9}, $H$ has a Hopf quotient
$\Hbar$ isomorphic to either $k[t^{\pm1}]$ with $t$ grouplike or $k[t]$
with $t$ primitive. In the first case, Theorems \ref{yythm6.3} and
\ref{yythm7.6} show that
$H$ is isomorphic to a Hopf algebra of type (Ib), (III) (with $q\neq 1$), 
or (IV). In the second, Propositions \ref{yyprop9.4} and
\ref{yyprop9.7} imply that $H_0={}_0H$ and either $H_0= k[y^{\pm1}]$ with
$y$ grouplike or $H_0= k[y]$ with $y$ primitive. Theorems \ref{yythm10.3}
and \ref{yythm11.3} then show that $H$ is isomorphic to a Hopf algebra of
type (IIb) or (V).
\hfill{$\square$}

\section{General properties: Proof of Proposition \ref{yyprop0.2}}
\label{yysec13}

The Hopf algebras classified in Theorem \ref{yythm0.1} enjoy a number of
common properties, which we summarized in Proposition \ref{yyprop0.2}. 
The {\it PI-degree}, denoted by $\pideg(A)$, of 
an algebra $A$ satisfying a polynomial identity is defined to the 
minimal integer $n$ such that $A$ can be embedded in an 
$n\times n$-matrix algebra $M_n(C)$ over a commutative ring $C$. 
(If $A$ does not satisfy a polynomial identity, then $\pideg(A)=\infty$.) 
Definitions of other relevant terms can be found, for instance, 
in \cite{BG, GLn, GW, Mon, MR}. 

\begin{proof}[Proof of Proposition \ref{yyprop0.2}]
(a) As an algebra, $H$ can be written as either a
skew-Laurent ring $R[x^{\pm1};\sigma]$ or a differential operator ring
$R[x;\delta]$, where $R$ is a commutative noetherian domain of Krull
dimension $1$. We claim that in the first case, $R$ always has a height $1$
prime ideal whose $\sigma$-orbit is finite, and that in the second, $R$ has
a height $1$ prime ideal which is $\delta$-stable. It then follows from
\cite[Theorem 6.9.13]{MR} that $\Kdim H=2$.

In case (I) (following the numbering of Theorem \ref{yythm0.1}), we can
write $H=R[x^{\pm1};\sigma]$ with
$R= k[y^{\pm1}]$ and either $\sigma(y)=y$ or $\sigma(y)=y^{-1}$. In either
case, the $\sigma$-orbit of the height $1$ prime $\langle y-1 \rangle$
contains at most two points. For case (III), we have $H=R[x^{\pm1};\sigma]$
with $R=k[y]$ and $\sigma(y)=qy$. Here, the height $1$ prime $\langle
y\rangle$ is $\sigma$-stable. Case (IV) has the form $H=R[x^{\pm1};\sigma]$
with $R= k[y_1,\dots,y_s]$ and $\sigma(y_i)= q^{m_i}y_i$ for all $i$. In
this case, $\langle y_1,\dots,y_s\rangle$ is a $\sigma$-stable height $1$
prime of $R$.

Turning to case (II), we have $H= R[x;\delta]$ with $R= k[y]$ and either
$\delta(y)=0$ or $\delta(y)=y$. In either case, $\langle y\rangle$ is a
$\delta$-stable height $1$ prime. In case (V), $H= R[x;\delta]$ with $R=
k[y^{\pm1}]$ and $\delta(y)= y^n-y$. Here, $\langle y-\lambda \rangle$ is a
$\delta$-stable height $1$ prime for any $(n-1)$st root of unity $\lambda
\in \kx$. 

This establishes the above claim, and thus $\Kdim H=2$. Moreover, when
$\gldim R=1$, it follows from \cite[Theorem 7.10.3]{MR} that $\gldim H=2$.
This covers cases (I), (II), (III), (V). In case (IV), $\gldim H =\infty$, as
already noted in the proof of Proposition \ref{yyprop1.6}.

(b) By \cite[Theorems 0.1, 0.2]{WZ1}, every affine noetherian PI Hopf algebra
is Aus\-land\-er-Gorenstein and GK-Cohen-Macaulay. In particular, this holds
for cases (I) and (IV), and the commutative subcase of (IIa). In these
cases, the GK-Cohen-Macaulay condition implies that the trivial module
$_Hk$ has grade $2$, and hence $\Ext^2_H({}_Hk,{}_HH) \ne 0$. Since $H$ is
AS-Gorenstein, it follows that $\injdim H =2$. In the remaining cases,
$\gldim H=2$, and we get $\injdim H=2$ because $H$ is noetherian. 

In the noncommutative subcase of (IIb), $H=k[y][x;yd/dy]$. Since $k[y]$ is
Auslander-regular, GK-Cohen-Macaulay, and connected graded, the
desired properties pass to $H$ by
\cite[Theorem I.15.3, Lemma I.15.4(a)]{BG}.

For case (III), view $H$ as the localization of $A= k[x][y;\sigma]$ with
respect to the regular normal element $x$, where $\sigma(x)=qx$. As in the
previous paragraph, $A$ is Auslander-regular and GK-Cohen-Macaulay by
\cite[Theorem I.15.3, Lemma I.15.4(a)]{BG}. These properties pass to $H$ by
\cite[Lemma II.9.11(b)]{BG}.

In case (V), finally, observe that $H\cong A/\langle yz-1\rangle$ where
$$A= k[y,z][x;(y^n-y)(yz\partial/\partial y- z^2\partial/ \partial z)].$$
As above, $A$ is Auslander-regular and GK-Cohen-Macaulay by
\cite[Theorem I.15.3, Lemma I.15.4(a)]{BG}. Since $yz-1$ is a central
regular element in $A$, \cite[Lemma I.15.4(b)]{BG} implies that $H$ is
Auslander-Gorenstein and GK-Cohen-Macaulay. (Actually, $H$ is
Auslander-regular, because
$\gldim H=2$.)

(c) Normal separation holds trivially in the commutative case, and it holds
for rings module-finite over their centers by \cite[Proposition 9.1]{GW}.
This covers cases (I) and (IV), the commutative subcase of (IIa), and the
subcase of (III) when $q$ is a root of unity. The solvable enveloping algebra
in case (IIb) is well known to have normal separation \cite[p. 217, second
paragraph]{GW}.

Next, consider case (III), with $q$ not a root of unity. Then $y$ is normal
in $H$, and the localization
$H[y^{-1}]$, a generic quantum torus, is a simple ring \cite[Corollary
1.18]{GW}. It follows that all nonzero prime ideals of $H$ contain $y$.
Since $H/\langle y\rangle$ is commutative, normal separation follows.

Case (V) remains. By \cite[Proposition 2.1]{GW}, the localization
$k(y)[x;\delta]$ of $H$ is simple, and hence all nonzero prime ideals of
$H$ have nonzero contractions to $k[y^{\pm1}]$. On the other hand, any
prime of $H$ contracts to a prime $\delta$-ideal of $k[y^{\pm1}]$, by
\cite[Theorem 3.22]{GW}. Since $\delta(y)= y^n-y$, the nonzero prime
$\delta$-ideals of $k[y^{\pm1}]$ are the ideals $\langle y-\lambda \rangle$
for $\lambda\in\kx$ with $\lambda^{n-1}=1$. For any such $\lambda$, the
commutator $[x,y-\lambda] = y^n-y$ is a multiple of $y-\lambda$, from which
we see that $y-\lambda$ is normal in $H$. Moreover, $H/\langle y-\lambda
\rangle$ is commutative. As in the previous case, we conclude that $\Spec
H$ has normal separation.

(d) This follows from part (c) by \cite[Theorem 12.17]{GW}.

(e) The algebra $H$ is affine as well as noetherian by Theorem
\ref{yythm0.1}, Auslander-Gorenstein and GK-Cohen-Macaulay with finite
GK-dimension by parts (a),(b) above, and $\Spec H$ has normal separation by
part (c). Therefore part (e) follows from \cite[Theorem 1.6]{GLn}.

(f) In all of our cases, $H$ can be viewed as a constructible algebra in the
sense of \cite[\S9.4.12]{MR}, and therefore $H$ satisfies the
Nullstellensatz \cite[Theorem 9.4.21]{MR}. It follows that locally
closed prime ideals of $H$ are primitive, and primitive ideals of $H$ are
rational \cite[Lemma II.7.15]{BG}. On the other hand, when $H$ is a
PI-algebra, Posner's theorem implies that rational primes of $H$ are
maximal, and therefore locally closed. Thus, we obtain the Dixmier-Moeglin
equivalence in cases (I) and (IV), and in case (III) when $q$ is a root of
unity. (We have given the proof in this manner because we could not locate
a reference in the literature for the relatively well known fact that
affine noetherian PI-algebras satisfy this equivalence.)

In Dixmier's development of the equivalence for enveloping algebras, the
solvable case is covered in \cite[Theorem 4.5.7]{Di}. This gives us the
equivalence in case (II).

For case (III), observe that the torus $(\kx)^2$
acts rationally on $H$ by
$k$-algebra automorphisms, such that $\langle 0\rangle$ and $\langle
y\rangle$ are the only primes of $H$ stable under the action.
Since we also have the Nullstellensatz, the Dixmier-Moeglin equivalence
holds by \cite[Theorem II.8.4]{BG}.

Finally, consider case (V). Our work in part (c) above shows that all
nonzero prime factors of $H$ are commutative. Thus, all nonzero rational
prime ideals of $H$ are maximal, and thus locally closed. It only remains
to note that the zero ideal is locally closed, because any nonzero prime
contains one of the finitely many primes $\langle y-\lambda \rangle$
for $\lambda\in\kx$ with $\lambda^{n-1}=1$.

(g) In all cases of Theorem \ref{yythm0.1}, $H$ is generated as a 
$k$-algebra by grouplike and skew primitive elements. It thus follows 
from \cite[Lemma 5.5.1]{Mon} that $H$ is pointed. 

(h) By Proposition 
\ref{yyprop3.4}(b) $\Ext^1_H(_Hk,_Hk)$ is either 1-dimensional 
or 2-di\-men\-sion\-al. If $\dim \Ext^1_H(_Hk,_Hk)$ is 2 which is 
the GK-dimension of $H$, then Proposition \ref{yyprop3.6} says 
that $H$ is commutative. Therefore $\Ext^1_H(_Hk,_Hk)$ is 
1-dimensional if and only if $H$ is not commutative.
\end{proof}

\begin{remark}
\label{yyrem13.1} 
One might have expected these Hopf algebras to satisfy other
properties such as unique factorization. However, any unique factorization
ring in the sense of Chatters and Jordan \cite{CJ} is a maximal order
\cite[Theorem 2.4]{CJ}, and the Hopf algebras $B$ of Construction
\ref{yycon1.2} are not -- simply observe that $B$ is contained
in the larger, and equivalent, order $k[y][x^{\pm1};\sigma]$. Thus, $B$ is
not a unique factorization ring.
\end{remark}

\begin{remark} 
\label{yyrem13.2}
Our original expectation was that for the Hopf algebras $H$ in 
Theorem \ref{yythm0.1}, the \emph{integral order} $\io(H)$ 
\cite[Definition 2.2]{LWZ} would equal the PI-degree of $H$. 
It can be shown that $\io(H)=\pideg(H)$ in cases (I), (II), (III), 
and (V). In case (IV), however, it turns out that $\pideg(H)=\ell$ 
while $\io(H)= \ell/\gcd(d,\ell)$ where
$$d := \ell+m(s-1)- \sum_{i=1}^s m_i$$
(in the notation of Construction \ref{yycon1.2}). For example, if 
$H= B(7,1,3,5,q)$, then $\pideg(H)= \ell= 105$ while $d= 112$ and 
so $\io(H)= 15$. 

In any case, $\io(H)$ divides $\pideg(H)$.
\end{remark}

\section*{Acknowledgement} 
We thank K.A.~Brown for helpful comments, suggestions, and discussions.

\providecommand{\MR}{\relax\ifhmode\unskip\space\fi MR }
\providecommand{\MRhref}[2]{%
   \href{http://www.ams.org/mathscinet-getitem?mr=#1}{#2} }
\providecommand{\href}[2]{#2}

\end{document}